\PassOptionsToPackage{hypertexnames=false}{hyperref}
\documentclass[12pt]{amsart}

\usepackage[numbers,sort&compress]{natbib}
\usepackage{amssymb,mathtools}
\usepackage{tikz}
\usetikzlibrary{arrows.meta}
\usepackage{microtype}
\usepackage{xcolor,xurl}
\usepackage{hyperref}
\hypersetup{
    linktoc=page,
    linkcolor=red,          % color of internal links
    citecolor=blue,        % color of links to bibliography
    filecolor=blue,      % color of file links
    urlcolor=cyan,
    colorlinks=true           % color of external links
}

\numberwithin{equation}{section}

\theoremstyle{plain}
\newtheorem{theorem}{Theorem}[section]
\newtheorem{proposition}[theorem]{Proposition}
\newtheorem{lemma}[theorem]{Lemma}

\theoremstyle{definition}

\newtheorem{remark}[theorem]{Remark}

\newcommand{\E}{\mathbb E}
\newcommand{\Pp}{\mathbb P}

\newcommand{\ind}{\mathbf 1}
\newcommand{\dd}{\,\mathrm d}
\newcommand{\R}{\mathbb R}
\newcommand{\stle}{\le_{\mathrm{st}}}

\begin{document}
\title[Distance and resistance on random series--parallel graphs]{Distance and resistance on random series--parallel graphs:
logarithmic speeds and near-critical asymptotics}
\author[Ding et~al.]{Ruiqi Ding}
\author[]{Zehua He}
\author[]{Yutao Liang}
\address[Ruiqi Ding, Zehua He, and Yutao Liang]{State Key Laboratory of Mathematical Sciences, Academy of Mathematics and
Systems Science, Chinese Academy of Sciences, Beijing 100190, China; School of
Mathematical Sciences, University of Chinese Academy of Sciences, Beijing 100049, China}
\email[Ruiqi Ding]{dingruiqi@amss.ac.cn}
\email[Zehua He]{hezehua@amss.ac.cn}
\email[Yutao Liang]{liangyutao@amss.ac.cn}
\author[]{Yushu Zheng}
\address[Yushu Zheng]{School of Mathematics, Shanghai University of Finance and Economics,
Shanghai 200433, China}
\email{yszheng666@gmail.com}

\begin{abstract}
We study the graph distance $D_n(p)$ and effective resistance
$R_n(p)$ between the boundary vertices of a depth-$n$ random series--parallel
graph, obtained by recursively joining two independent copies in series with
probability $p$ and in parallel with probability $1-p$. We prove the
existence of deterministic logarithmic speeds: for every $p\in[0,1]$,
$n^{-1}\log D_n(p)$ and $n^{-1}\log R_n(p)$ converge to deterministic limits
$v_D(p)$ and $v_R(p)$, respectively, almost surely and in $L^1$. The
limiting speeds agree with the corresponding first-moment logarithmic rates
for every $p\in[0,1]$ in the distance case and for $p\in[1/2,1]$ in the
resistance case.
We further determine the near-critical behavior of the resistance speed:
$v_R(\frac12+\delta)\sim
2\zeta(3)^{1/3}\lambda_*\delta^{2/3}$ as $\delta\downarrow0$, where
$\lambda_*>0$ is characterized by an explicit nonlinear boundary-value
problem. This exponent $2/3$ contrasts with the exponent $1/2$ for distance
obtained by Chen, Derrida, Duquesne, and Shi (2026). The main idea of this work was proposed by ChatGPT 5.6 Sol, and the authors take full responsibility for the mathematical content. The three main theorems and their supporting proof dependencies have been formalized in Lean 4, relative to two explicitly documented external mathematical inputs.
\end{abstract}
\maketitle
\hypersetup{
  pdftitle={Distance and resistance on random series--parallel graphs:
    logarithmic speeds and near-critical asymptotics},
  pdfauthor={Ruiqi Ding; Zehua He; Yutao Liang; Yushu Zheng},
  pdfsubject={Stochastic processes on random series--parallel graphs},
  pdfkeywords={random series--parallel graph; graph distance; effective
    resistance; logarithmic speeds; near-critical asymptotics}
}

\section{Introduction}
\label{sec:introduction}

Hierarchical random systems provide a natural setting for studying the
competition between recursive mechanisms acting in opposite directions.
Their balance may give rise to a critical regime, while a small bias away
from balance can be amplified across scales.  A natural question is then how
the system departs from its critical behavior when the balance is weakly
broken, and how the scale of this departure depends on the strength of the
bias.

The random series--parallel graph introduced by Hambly and
Jordan~\cite{HamblyJordan2004} provides a concrete instance of this
phenomenon.  Starting from a single edge, at each generation every edge is
independently replaced by two edges in series with probability $p$, and by
two edges in parallel with probability $1-p$.  Let $D_n(p)$ and $R_n(p)$
denote, respectively, the graph distance and the effective resistance
between the two boundary vertices after $n$ generations.  Hambly and
Jordan~\cite[Theorem~2.1 and Lemma~3.1]{HamblyJordan2004}
proved that, almost surely, both $D_n(p)$ and $R_n(p)$ tend to infinity when
$p>\frac12$, whereas when $p<\frac12$, $D_n(p)$ converges to a finite limit
and $R_n(p)$ tends to zero.

These phase-transition results describe only the qualitative long-time
behavior of $D_n(p)$ and $R_n(p)$.  Since the same random operation is
iterated at each generation, it is natural to ask whether their logarithms
are asymptotically linear in the generation number.  This leads to the
following questions:
\begin{enumerate}
\item For every fixed $p\in[0,1]$, do $n^{-1}\log D_n(p)$ and
$n^{-1}\log R_n(p)$ have deterministic almost-sure limits?  Do
$n^{-1}\log\E D_n(p)$ and $n^{-1}\log\E R_n(p)$ also converge?  When the
relevant limits exist, how are the almost-sure logarithmic speeds and the
first-moment logarithmic rates related?
\item If these logarithmic speeds exist, what are their near-critical
asymptotics as $p$ approaches the critical value $\frac12$?
\end{enumerate}

For the first problem, Hambly and Jordan~\cite[Lemma 3.1(a)]{HamblyJordan2004} established
the almost-sure exponential growth of $D_n(p)$ for $p>1/2$.  Chen, Derrida,
Duquesne, and
Shi~\cite[equation~(1.2)]{ChenDerridaDuquesneShi2026} later proved that the
first-moment logarithmic rate
\[
 \lim_{n\to\infty}\frac1n\log\E D_n(p)
\]
exists for every $p\in[0,1]$ and is positive exactly when $p>1/2$.  The latter
work did not establish the corresponding almost-sure convergence for all
$p$.  For resistance, Hambly and Jordan~\cite[Corollaries 2.1 and 2.2]{HamblyJordan2004} obtained almost-sure exponential
upper and lower bounds for $p\ne1/2$ and proved that
$n^{-1}\log R_n(1/2)\to0$ almost surely,
but left the general convergence question open.

It is also worth noting that, at criticality, more precise sublinear behavior
is known for both quantities.  Auffinger and
Cable~\cite[Theorem~1]{AuffingerCable2017} proved that
$\log D_n(1/2)$ has scale $\sqrt n$ and identified its limiting law.  For
resistance, Addario-Berry et al.~\cite{AddarioBerryEtAl2020} predicted an
$n^{1/3}$-scale limit law for $\log R_n(1/2)$; this was recently proved
independently by Chen, Duquesne, and Shi~\cite{ChenDuquesneShi2026} and
Morfe~\cite{Morfe2026}.

For the second problem, Chen, Derrida, Duquesne, and
Shi~\cite[Theorem~1]{ChenDerridaDuquesneShi2026} obtained the sharp
near-critical asymptotics of this first-moment logarithmic rate for distance,
showing that it vanishes on the scale $(p-\frac12)^{1/2}$ as
$p\downarrow\frac12$. For resistance, Morfe~\cite[Theorem~2 and Section~1.5]{Morfe2026} established a PDE scaling limit for the \(N\)-dependent bias parameter \(p^{(N)}=\frac12+\theta N^{-1}\) and related the limiting equation to a reaction--diffusion travelling-wave problem. This continuum picture motivates our near-critical analysis but does not by itself yield the resistance speed at fixed \(p\ne\frac12\) or its near-critical asymptotics.

Combined with these earlier results, our work answers the two
questions above.  We prove that, for every $p\in[0,1]$,
$n^{-1}\log D_n(p)$ and $n^{-1}\log R_n(p)$ converge almost surely and in
$L^1$ to deterministic limits.  We also prove that
$n^{-1}\log\E D_n(p)$ and $n^{-1}\log\E R_n(p)$ converge.  The distance rate
agrees with $v_D(p)$ throughout $[0,1]$, while the resistance rate is
$v_R(p)\vee\log(2p)$ for $0\le p\le1$, with the maximum equal to $v_R(p)$
for $1/2\le p\le1$, and to $\log(2p)$ for $p<1/2$ sufficiently close to
$1/2$.  We further determine the sharp near-critical asymptotics of the
resistance speed, showing that its magnitude vanishes on the scale
$\lvert p-\frac12\rvert^{2/3}$ as $p\to\frac12$.

\subsection{Main results}
\label{sec:main-results}

We define the random series--parallel graph through a binary-tree
construction.  Let
\[
 \mathbb T:=\bigcup_{k\ge0}\{0,1\}^k,
 \qquad \{0,1\}^0:=\{\varnothing\},
\]
and let $(\xi_u)_{u\in\mathbb T}$ be independent Bernoulli random variables
with parameter $p$.  The initial graph $\mathcal G_0(p)$ consists of a
single edge $e_{\varnothing}$ joining two distinguished vertices $s$ and
$t$, called the boundary vertices.  Suppose recursively that the edges of
$\mathcal G_k(p)$ are indexed by $u\in\{0,1\}^k$.  For each edge $e_u$ with
endpoints $x_u$ and $y_u$, perform the following replacement.  If
$\xi_u=1$, introduce a new vertex $z_u$ and replace $e_u$ by the two edges
\[
 e_{u0}=\{x_u,z_u\},
 \qquad
 e_{u1}=\{z_u,y_u\}
\]
in series.  If $\xi_u=0$, replace $e_u$ by two distinct parallel edges
$e_{u0}$ and $e_{u1}$, both joining $x_u$ and $y_u$.  Performing these
replacements simultaneously for all $u\in\{0,1\}^k$ produces
$\mathcal G_{k+1}(p)$.  Figure~\ref{fig:series-parallel-replacement}
illustrates the first two refinement steps for one realization of the
environment.

\begin{figure}
\centering
\begin{tikzpicture}[
  line width=0.7pt,
  every node/.style={font=\small},
  vertex/.style={circle,fill=black,inner sep=1.7pt},
  transition/.style={-{Latex[length=2.2mm,width=1.6mm]},draw=black!70,
    line width=0.6pt,shorten <=20pt,shorten >=20pt}
]
  
  \node at (-6.25,1.15) {$\mathcal G_0(p)$};
  \node[vertex,label=below:$s$] (g0s) at (-6.9,0) {};
  \node[vertex,label=below:$t$] (g0t) at (-5.6,0) {};
  \draw (g0s) -- node[above] {$e_{\varnothing}$} (g0t);

  \node at (-1.95,1.15) {$\mathcal G_1(p)$};
  \node[vertex,label=below:$s$] (g1s) at (-3.25,0) {};
  \node[vertex,label=below:$z_{\varnothing}$] (g1z) at (-1.95,0) {};
  \node[vertex,label=below:$t$] (g1t) at (-0.65,0) {};
  \draw (g1s) -- node[above] {$e_0$} (g1z)
               -- node[above] {$e_1$} (g1t);
  \draw[transition] (g0t) --
    node[midway,above,yshift=2pt,font=\footnotesize,text=black]
    {$\xi_{\varnothing}=1$} (g1s);

  \node at (4.55,1.15) {$\mathcal G_2(p)$};
  \node[vertex,label=below:$s$] (g2s) at (2.25,0) {};
  \node[vertex,label=below:$z_{\varnothing}$] (g2z0) at (4.55,0) {};
  \node[vertex,label=below:$z_1$] (g2z1) at (5.7,0) {};
  \node[vertex,label=below:$t$] (g2t) at (6.85,0) {};
  \draw (g2s) .. controls (2.85,0.65) and (3.95,0.65) ..
    node[above] {$e_{00}$} (g2z0);
  \draw (g2s) .. controls (2.85,-0.65) and (3.95,-0.65) ..
    node[below] {$e_{01}$} (g2z0);
  \draw (g2z0) -- node[above] {$e_{10}$} (g2z1)
                -- node[above] {$e_{11}$} (g2t);
  \draw[transition] (g1t) --
    node[midway,above,yshift=2pt,font=\footnotesize,text=black]
    {$(\xi_0,\xi_1)=(0,1)$} (g2s);
\end{tikzpicture}
\caption{Two successive refinement steps in one realization of the
environment.  The choice $\xi_{\varnothing}=1$ replaces the initial edge in
series.  At the next level, $\xi_0=0$ replaces $e_0$ in parallel, while
$\xi_1=1$ replaces $e_1$ in series.}
\label{fig:series-parallel-replacement}
\end{figure}
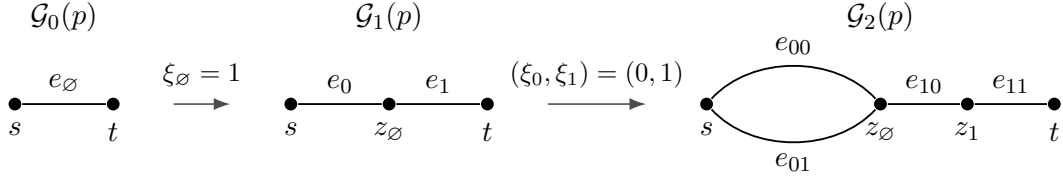

At every level, $s$ and $t$ remain the distinguished boundary vertices.
Let $D_n(p)$ be the graph distance between $s$ and $t$ in
$\mathcal G_n(p)$, and let $R_n(p)$ be their effective resistance when every
edge has unit resistance.

\begin{theorem}[Logarithmic speeds]
\label{thm:logarithmic-speeds}
For every $p\in[0,1]$, there exist deterministic constants $v_D(p)$ and
$v_R(p)$ such that
\begin{equation}
 \frac1n\log D_n(p)\longrightarrow v_D(p),
 \qquad
 \frac1n\log R_n(p)\longrightarrow v_R(p)
 \qquad\text{a.s. and in }L^1.
 \label{eq:full-range-L1}
\end{equation}
Their values satisfy
\begin{equation*}
 v_D(p)=0\quad\text{for }0\le p\le\frac12,
 \qquad
 v_R(1-p)=-v_R(p)\quad\text{for }0\le p\le1,
\end{equation*}
so in particular $v_R(1/2)=0$.  Moreover, for $1/2<p\le1$,
\begin{equation}
 v_D(p),\,v_R(p)\in[\log(2p),\log2].
 \label{eq:supercritical-speed-bounds}
\end{equation}
\end{theorem}

\begin{theorem}[First-moment logarithmic rates]
\label{thm:first-moment-logarithmic-rates}
For every $p\in[0,1]$, the limits
\begin{equation*}
 \gamma_D(p):=\lim_{n\to\infty}\frac1n\log\E D_n(p),
 \qquad
 \gamma_R(p):=\lim_{n\to\infty}\frac1n\log\E R_n(p)
\end{equation*}
exist.  With the convention $\log0=-\infty$, they are related to the
logarithmic speeds by
\begin{equation}\label{eq:resistance-first-moment-dichotomy}
 \gamma_D(p)=v_D(p),
 \qquad
 \gamma_R(p)=v_R(p)\vee\log(2p)
 \qquad\text{for }0\le p\le1.
\end{equation}
In particular,
\begin{equation}
 \gamma_R(p)=v_R(p)\qquad\text{for }\frac12\le p\le1.
 \label{eq:supercritical-resistance-identification}
\end{equation}
\end{theorem}

Figure~\ref{fig:resistance-speed-comparison} schematically illustrates
the maximum in \eqref{eq:resistance-first-moment-dichotomy}.

\begin{figure}
\centering
\begin{tikzpicture}[
  scale=0.7,
  transform shape,
  x=6.8cm,
  y=2.35cm
]
  \def\logtwo{0.693147}

  \draw[gray!25,line width=0.35pt] (0,-\logtwo) -- (1,-\logtwo);
  \draw[gray!25,line width=0.35pt] (0, \logtwo) -- (1, \logtwo);
  \draw[gray!55,densely dotted,line width=0.45pt]
    (0.5,-1.35) -- (0.5,1.05);

  \begin{scope}
    \clip (-0.02,-1.35) rectangle (1.04,1.05);
    \draw[black!85,dashed,line width=0.9pt,domain=0.02:1,
      samples=180,smooth,variable=\p]
      plot ({\p},{ln(2*\p)});
    \draw[blue!70!black,line width=1.15pt,domain=0:0.5,
      samples=100,smooth,variable=\p]
      plot ({\p},{-ln(2)*pow(1-2*\p,2/3)
        -0.30*pow(1-2*\p,2/3)*(2*\p)});
    \draw[blue!70!black,line width=1.15pt,domain=0.5:1,
      samples=100,smooth,variable=\p]
      plot ({\p},{ln(2)*pow(2*\p-1,2/3)
        +0.30*pow(2*\p-1,2/3)*(2-2*\p)});
  \end{scope}

  \draw[line width=0.55pt] (-0.02,0) -- (1.04,0);
  \draw[line width=0.55pt] (0,-1.35) -- (0,1.05);
  \foreach \x/\lab in
    {0/0,0.25/{\frac14},0.5/{\frac12},0.75/{\frac34},1/1}
    {
      \draw[line width=0.45pt] (\x,0.025) -- (\x,-0.025);
      \node[below=6pt,fill=white,inner sep=0pt] at (\x,-0.025)
        {$\lab$};
    }
  \foreach \y/\lab in {-\logtwo/{-\log 2},\logtwo/{\log 2}}
    {
      \draw[line width=0.45pt] (-0.008,\y) -- (0.008,\y);
      \node[left=2pt] at (-0.008,\y) {$\lab$};
    }
  \node[below right=1pt] at (1.02,0) {$p$};
  \fill[blue!70!black] (0.5,0) circle[radius=1.5pt];

  \draw[blue!70!black,line width=1.15pt] (1.13,0.68) -- (1.22,0.68);
  \node[right=3pt] at (1.22,0.68) {$v_R(p)$};
  \draw[black!85,dashed,line width=0.9pt] (1.13,0.53) -- (1.22,0.53);
  \node[right=3pt] at (1.22,0.53) {$\log(2p)$};
\end{tikzpicture}
\caption{Schematic comparison of the two terms in
$\gamma_R(p)=v_R(p)\vee\log(2p)$.  For $1/2\le p\le1$,
Theorem~\ref{thm:first-moment-logarithmic-rates} gives
$v_R(p)\ge\log(2p)$, so the maximum is $v_R(p)$.  For $p<1/2$
sufficiently close to $1/2$, Theorem~\ref{thm:near-critical-speed} gives
$\log(2p)>v_R(p)$, so the maximum is $\log(2p)$.  No claim is made
about the global curvature of $v_R$ or the number
or locations of crossings between $v_R(p)$ and $\log(2p)$ in
$0<p<1/2$. }
\label{fig:resistance-speed-comparison}
\end{figure}
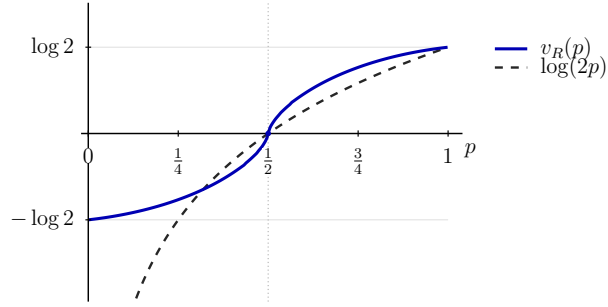

\begin{theorem}[Resistance speed near criticality]
\label{thm:near-critical-speed}
For $\lambda>0$, consider the boundary-value problem
\begin{equation}
\begin{aligned}
 &W^2W'-\lambda W+u(1-u)=0,\qquad 0<u<1,\\
 &W(u)>0\quad\text{for }\ 0<u<1, \qquad
 W(0)=W(1)=0.
\end{aligned}
 \label{eq:main-W-bvp}
\end{equation}
We call $\lambda>0$ \emph{admissible} if \eqref{eq:main-W-bvp} admits a solution
$W\in C([0,1])\cap C^1((0,1))$.
Then there exists $\lambda_*\in(0,\infty)$ such that $\lambda$ is admissible
if and only if $\lambda\ge\lambda_*$.   Moreover, as
$\delta\downarrow0$, we have
\begin{equation}
 v_R\left(\frac12+\delta\right)
 =-v_R\left(\frac12-\delta\right)
 \sim2\zeta(3)^{1/3}\lambda_*\,\delta^{2/3},
 \label{eq:main-near-critical-limit}
\end{equation}
and
\begin{equation}
 \gamma_R\left(\frac12+\delta\right)
 \sim2\zeta(3)^{1/3}\lambda_*\,\delta^{2/3},
 \qquad
 \gamma_R\left(\frac12-\delta\right)\sim-2\delta.
 \label{eq:resistance-first-moment-near-critical}
\end{equation}
\end{theorem}

\begin{remark}
For comparison, Chen, Derrida, Duquesne, and
Shi~\cite[equation~(1.2) and Theorem~1]{ChenDerridaDuquesneShi2026} proved
that the first-moment logarithmic rate for distance exists and obtained its
sharp
near-critical asymptotics.  Together with
Theorem~\ref{thm:first-moment-logarithmic-rates}, their results give
\begin{equation*}
 v_D\left(\frac12+\delta\right)=\gamma_D\left(\frac12+\delta\right)
 \sim \sqrt{\zeta(2)}\,\delta^{1/2}
 =\frac{\pi}{\sqrt6}\,\delta^{1/2}
 \qquad\text{as }\delta\downarrow0.
\end{equation*}
Thus the distance and resistance speeds vanish with different critical
exponents $1/2$ and $2/3$, respectively.

\end{remark}

\subsection{Proof outline}

We next outline the main ideas behind the proofs of
Theorems~\ref{thm:logarithmic-speeds}--\ref{thm:near-critical-speed}.

To prove Theorems~\ref{thm:logarithmic-speeds}
and~\ref{thm:first-moment-logarithmic-rates}, we first establish the
existence of the first-moment logarithmic rates.  The argument uses the conditional refinement property in
Proposition~\ref{prop:conditional-refinement}: conditional on
$\mathcal G_n(p)$, the network $\mathcal G_{n+r}(p)$ is obtained by replacing
each edge of $\mathcal G_n(p)$ with an independent copy of $\mathcal G_r(p)$.  For
distance, we replace each edge of a shortest path in $\mathcal G_n(p)$ by a
shortest path across the corresponding depth-$r$ copy; the resulting
candidate path has conditional expected length
$(\E D_r(p))D_n(p)$.  For resistance, let $(\theta_e)$ be a minimizing unit
flow on $\mathcal G_n(p)$ and, in the copy replacing $e$, route $\theta_e$
times a minimizing unit flow.  The resulting candidate flow has conditional
expected energy $(\E R_r(p))R_n(p)$.  Taking expectations gives first-moment
submultiplicativity, and Fekete's subadditive lemma then gives the desired rates.  See
Lemma~\ref{lem:first-moment-submultiplicativity} for details.

The main remaining issue is to show that, for both
$D_n(p)$ and $R_n(p)$, the gap between the logarithm of the mean and the mean
logarithm, which we call the Jensen gap, is $o(n)$. For $p>1/2$, on the respective events
$\{D_n(p)\ge\E D_n(p)/2\}$ and $\{R_n(p)\ge\E R_n(p)/2\}$, the corresponding
Jensen gaps are bounded by $\log2$ plus the absolute fluctuations of the
logarithms around their means, as shown in
\eqref{eq:jensen-event-comparison}.  Lemma~\ref{lem:center-tracking} bounds
the expectations of these fluctuations uniformly in $n$.  The two events
have probabilities uniformly bounded away from zero by the second-moment
method; see Lemma~\ref{lem:normalized-L2} and
\eqref{eq:Paley-Zygmund}.  This identifies the first-moment logarithmic rates
with the logarithmic speeds.
The fluctuation control in Lemma~\ref{lem:center-tracking} upgrades convergence of the mean logarithms
to the almost-sure and $L^1$ convergence in \eqref{eq:full-range-L1}.

To prove Theorem~\ref{thm:near-critical-speed}, we construct two families of random variables that yield upper and lower bounds for $\log R_n(\frac12+\delta)$. Consider the upper bound, for example. Let $\mathcal T_{\frac12+\delta}$ denote the one-step cumulative distribution function (CDF)
update of $\log R_n(\frac12+\delta)$ (see Section~\ref{sec:cdf-operator}), and write $\delta=\varepsilon^3$. For any
\(\lambda>\lambda_*\) and
\(\kappa>\kappa_\lambda:=2\zeta(3)^{1/3}\lambda\), and all sufficiently
small \(\varepsilon\), we construct a random variable
\(Y_\varepsilon\) with CDF \(G_\varepsilon\) such that
\begin{align}\label{prop_G}
 F_0\ge G_\varepsilon,
 \qquad
 \mathcal T_{1/2+\varepsilon^3}G_\varepsilon(x)
 \ge G_\varepsilon(x-\kappa\varepsilon^2).
\end{align}
Here $F_0(x)=\mathbf 1_{\{x\ge0\}}$ is the CDF of $\log R_0(\frac12+\varepsilon^3)$. The first inequality in \eqref{prop_G} gives $\log R_0(\frac12+\varepsilon^3)
 \stle Y_\varepsilon$. Then the second inequality in \eqref{prop_G}, combined with the basic properties of random series--parallel graphs (see Lemma~\ref{lem:exact-cdf}), gives
  $\log R_n(\frac12+\varepsilon^3)
 \stle Y_\varepsilon+n\kappa\varepsilon^2$, and hence
\(v_R(1/2+\varepsilon^3)\le\kappa\varepsilon^2\).
Letting \(\kappa\downarrow\kappa_\lambda\) and then
\(\lambda\downarrow\lambda_*\) gives the sharp upper bound. The matching
lower bound follows from the same reasoning.

To construct \(G_\varepsilon\) satisfying \eqref{prop_G}, we compare the  one-step change
\(\mathcal T_{1/2+\delta}G_\varepsilon(x)-G_\varepsilon(x)\) with the change
\(G_\varepsilon(x-\kappa\varepsilon^2)-G_\varepsilon(x)\) caused by translation.
 Matching their leading-order coefficients yields the
equation in Theorem~\ref{thm:near-critical-speed}:
\[
 W^2W'-\lambda W+u(1-u)=0.
\]
The self-contained analysis in Appendices~\ref{app:ode} and
\ref{app:regularity} provides estimates
needed to construct \(G_\varepsilon\).

\begin{remark}
    The differential equation in Theorem~\ref{thm:near-critical-speed} comes
directly from the above comparison between the CDF update and the translation. Test the operator on a slowly varying CDF
$F_{\varepsilon,z_0}(x)=\Phi(z_0+\varepsilon x)$, with
$q=\Phi'$.  Lemma~\ref{lem:exact-cdf} and
Lemma~\ref{lem:weighted-consistency} give, for all $x$ and $z_0$,
\[
 \mathcal T_{1/2+\varepsilon^3}F_{\varepsilon,z_0}(x)
 -F_{\varepsilon,z_0}(x)
 =\varepsilon^3\bigl(a q q'-2\Phi(1-\Phi)\bigr)
  +O(\varepsilon^5q).
\]
Here the quantities on the right are evaluated at
$z_0+\varepsilon x$.
A translation by $\kappa\varepsilon^2$ changes the same CDF by
$-\kappa\varepsilon^3q+O(\varepsilon^6q)$.  Matching the two leading terms
therefore yields the  equation
\[
 a q q'-2\Phi(1-\Phi)=-\kappa q.
\]
Here $a=2\zeta(3)$, as evaluated in Lemma~\ref{lem:diffusion-coefficient}. With $q=\beta W(1-\Phi)$, $\beta^3=2/a$, and
$\kappa=2\lambda/\beta$, this equation is exactly
$W^2W'-\lambda W+u(1-u)=0$.   Remark~\ref{rem:two-thirds-heuristic} records the corresponding
scaling balance.

\end{remark}

\subsection{Further related work}

The recursions studied here are connected with the broader literature on
nonlinear hierarchical systems, in which each generation is obtained by
applying a deterministic or random update map to independent copies of the
preceding generation.  General limit questions for such systems were studied
by Shneiberg~\cite{Shneiberg1986}, Li and Rogers~\cite{LiRogers1999}, and
Jordan~\cite{Jordan2002}.  For general background on recursive distributional
equations and their associated recursive tree processes, see Aldous and
Bandyopadhyay~\cite{AldousBandyopadhyay2005}.  A distinctive feature of the
present model is that its distributional recursion is accompanied by a
canonical coupling through the underlying series--parallel networks.

The exact CDF recursion underlying our near-critical analysis admits a
continuum interpretation: after rescaling, its leading-order behavior is
described by a continuum equation for candidate limiting shapes of the
rescaled CDFs.  Similar connections arise in hipster random walks and
cooperative-motion models
\cite{AddarioBerryEtAl2020,AddarioBerryBeckmanLin2022,
AddarioBerryBeckmanLin2024}, while Morfe~\cite{Morfe2026} developed this
approach for a class of recursive distributional equations that includes the
critical resistance recursion.  Our proof uses
the continuum equation
only to identify these candidate shapes and remains otherwise discrete.  We
use the resulting shapes to construct upper and lower comparison CDFs, verify
one-step inequalities directly for the exact recursion, and iterate them
using stochastic order.  The order preservation and approximation bounds
underlying this argument are reminiscent of the Barles--Souganidis framework
\cite{BarlesSouganidis1991}.  The threshold in the associated continuum
equation is analogous to minimum-speed thresholds in travelling-wave theory
\cite{HadelerRothe1975,SanchezGardunoMaini1994}.

\subsection{Organization and conventions}
\label{sec:organization-conventions}

The remainder of the paper is organized as follows.
Section~\ref{sec:preliminaries} collects the necessary preliminaries.
Sections~\ref{sec:ballistic-proof}
and~\ref{sec:two-sided-barriers} prove
Theorems~\ref{thm:logarithmic-speeds}--\ref{thm:first-moment-logarithmic-rates}
and Theorem~\ref{thm:near-critical-speed}, respectively.
Appendices~\ref{app:ode} and~\ref{app:regularity} establish the
ODE analysis used in the upper and lower bound arguments of Section~\ref{sec:two-sided-barriers}.

For real-valued random variables $U$ and $V$, we write
$U\stle V$ when $U$ is stochastically dominated by $V$.  We write
$x_+:=\max\{x,0\}$ unless otherwise specified. Unless stated otherwise,
$c$ and $C$ denote finite
positive constants whose values may change from line to line.

\subsection{AI-assisted proof development and formal verification}
\label{subsec:ai-methodology}

The main idea of this work was proposed by generative AI.
In particular, OpenAI ChatGPT 5.6 Sol and Codex
were used to generate initial drafts of the mathematical proofs and to generate
initial versions of the Lean~4 formalization.  The tools were also used
to compare the informal and formal statements and to identify missing
hypotheses, unused assumptions, and possible gaps in the dependency
structure.  No output of a generative-AI tool was treated as mathematical
evidence or accepted solely on the basis of its generation.

Every mathematical argument was subsequently reconstructed and checked
by the authors.  This verification included checking the quantifiers,
parameter ranges, limiting modes, constants, signs, and dependence of
each result on earlier lemmas; independently verifying cited results
against their original sources; and revising or replacing generated
arguments whenever the authors could not justify every step.  The final
mathematical statements, proofs, interpretations, and claims of novelty
were determined and approved by the authors, who take full responsibility
for their correctness.

The main results were additionally formalized in Lean~4.  The
formalization was checked by compiling the complete project with
\texttt{lake build}, searching the source tree for unfinished proof
constructs such as \texttt{sorry} and \texttt{admit}, inspecting the
axiom dependencies of the final theorem declarations, and comparing
each formal theorem with its numbered manuscript counterpart through a
source map.  The formalization is complete relative to the following
two explicitly isolated external mathematical inputs:
\(    \gamma_D(1/2)=0, \)
and compact-interval Peano existence for scalar ordinary differential
equations under continuity and a linear-growth bound.    All remaining project-specific
dependencies of the three main theorem declarations are proved within
the formal development.
\section{Preliminaries}
\label{sec:preliminaries}

This section collects the preliminary tools used in the proofs of our main
results.  Section~\ref{sec:gate-properties} gives a unified recursion for
distance and resistance and analyzes it on the logarithmic scale, including
the one-step change in the mean logarithm.
Section~\ref{sec:canonical-refinement-coupling} records some basic structural
properties.
Section~\ref{sec:cdf-operator} specializes to resistance and derives the exact
CDF update used in the near-critical analysis.

\subsection{Unified recursion and logarithmic update maps}
\label{sec:gate-properties}

Fix $p\in[0,1]$.
Recall from Section~\ref{sec:main-results} that, for this fixed $p$, the
two sequences $(D_n(p))_{n\ge0}$ and $(R_n(p))_{n\ge0}$ are jointly
constructed from the common environment $(\xi_u)_{u\in\mathbb T}$.
To treat distance and resistance simultaneously, set
\[
 Z_n^{(0)}(p):=D_n(p),
 \qquad
 Z_n^{(1)}(p):=R_n(p).
\]
For $\eta\in\{0,1\}$ and $x,y>0$, define the corresponding
parallel operation by
\begin{equation*}
\Pi_\eta(x,y)
 :=
 \frac{x\wedge y}
 {\left(1+\dfrac{x\wedge y}{x\vee y}\right)^\eta}.
\end{equation*}
Thus $\Pi_0(x,y)=x\wedge y$ and
$\Pi_1(x,y)=xy/(x+y)$.
Decomposing the environment tree at its root then yields the marginal
recursion:
\begin{equation}
 Z_{n+1}^{(\eta)}(p)
 \overset{\mathrm d}{=}
 \begin{cases}
 Z_n^{(\eta,1)}(p)+Z_n^{(\eta,2)}(p),&\text{with probability }p,\\[2mm]
 \Pi_\eta\!\left(Z_n^{(\eta,1)}(p),Z_n^{(\eta,2)}(p)\right),
     &\text{with probability }1-p.
 \end{cases}
 \label{eq:unified-rde}
\end{equation}
Here $Z_n^{(\eta,1)}(p)$ and $Z_n^{(\eta,2)}(p)$ are independent
copies of $Z_n^{(\eta)}(p)$ and are independent of the choice between
series and parallel composition at the root.

We next introduce the series and parallel update maps in logarithmic coordinates.
For $x,y\in\R$, define the
logarithmic series map and the logarithmic parallel map with parameter
$\eta$, respectively, by
\[
 g_+(x,y)=\log(e^x+e^y),
 \qquad
 g_-^{(\eta)}(x,y)=
 \begin{cases}
  x\wedge y,&\eta=0,\\
 -\log(e^{-x}+e^{-y}),&\eta=1.
 \end{cases}
\]
Thus $g_+(x,y)$ is the logarithm of either the graph distance or the effective
resistance obtained by joining two components in series whose corresponding
logarithmic quantities are $x$ and $y$.  Likewise, $g_-^{(0)}(x,y)$ and
$g_-^{(1)}(x,y)$ are, respectively, the logarithmic graph distance and
logarithmic effective resistance obtained by joining such components in
parallel.

The definitions immediately give the following properties:
\begin{enumerate}
\item The maps are nondecreasing in each coordinate.
\item The maps are translation covariant: for every $x,y,a\in\R$,
\begin{equation}
 g_+(x+a,y+a)=a+g_+(x,y),
 \qquad
 g_-^{(\eta)}(x+a,y+a)=a+g_-^{(\eta)}(x,y).
 \label{eq:gate-translation}
\end{equation}
\end{enumerate}

We shall use the following max--min representations of these maps.  Let
\[
 h(t)=\log(1+e^{-t}),\qquad t\ge0.
\]
Then
\begin{equation}
 g_+(x,y)=(x\vee y)+h(|x-y|),
 \qquad
 g_-^{(\eta)}(x,y)=(x\wedge y)-\eta h(|x-y|).
 \label{eq:log-series-gate}
\end{equation}

We now set
\[
 X_n^{(\eta)}(p):=\log Z_n^{(\eta)}(p).
\]
Its recursion and mean drift are summarized in the following lemma.

\begin{lemma}
\label{lem:logarithmic-drift}
Fix $\eta\in\{0,1\}$ and $p\in[0,1]$, and write
$X_n:=X_n^{(\eta)}(p)$.  For $n\ge0$, if $X_n'$ is an independent copy of
$X_n$, then
\begin{equation}
 X_{n+1}\overset{\mathrm d}{=}
 \begin{cases}
  g_+(X_n,X_n'),&\text{with probability }p,\\
  g_-^{(\eta)}(X_n,X_n'),&\text{with probability }1-p,
 \end{cases}
 \label{eq:log-rde}
\end{equation}
where the choice between the two cases is independent of $(X_n,X_n')$.
Moreover,
\begin{equation}
 \E X_{n+1}-\E X_n
 =\left(p-\frac12\right)\E|X_n-X_n'|
 +\{p-\eta(1-p)\}\E h(|X_n-X_n'|).
 \label{eq:drift-identity}
\end{equation}
\end{lemma}

\begin{proof}
The recursion \eqref{eq:log-rde} is immediate.  The identities
\[
 \E(X_n\vee X_n')=\E X_n+\frac12\E|X_n-X_n'|,
 \qquad
 \E(X_n\wedge X_n')=\E X_n-\frac12\E|X_n-X_n'|
\]
and the logarithmic maps in \eqref{eq:log-series-gate} give
\eqref{eq:drift-identity}.
\end{proof}

\subsection{Basic structural properties}
\label{sec:canonical-refinement-coupling}

Recall the binary-tree representation of the model from
Section~\ref{sec:main-results}.  We record several basic structural properties
that will be used below.  For $n\ge0$, let
\[
 \mathcal F_n:=\sigma(\xi_u:|u|<n).
\]

\begin{proposition}[Conditional refinement property]
\label{prop:conditional-refinement}
For every $n,r\ge0$, conditionally on $\mathcal F_n$, the graph
$\mathcal G_{n+r}(p)$ is obtained from $\mathcal G_n(p)$ by replacing each
edge with an independent depth-$r$ copy of $\mathcal G_r(p)$.
\end{proposition}

\begin{proof}
For distinct $u\in\{0,1\}^n$, the descendant environments
$(\xi_{uw})_{w\in\mathbb T}$ are independent of one another and of
$\mathcal F_n$, and each has the same law as the original environment.
Restricting these environments to $|w|<r$ gives the claim.
\end{proof}

\begin{lemma}
\label{lem:deterministic-bound}
For $\eta\in\{0,1\}$, $p\in[0,1]$, and $n\ge0$, we have
\begin{equation}
 2^{-\eta}Z_n^{(\eta)}(p)
 \le Z_{n+1}^{(\eta)}(p)
 \le 2Z_n^{(\eta)}(p)
 \qquad\text{almost surely}.
 \label{eq:adjacent-Z}
\end{equation}
Consequently,
\begin{equation}
\begin{aligned}
 -\eta n\log2
 \le X_n^{(\eta)}(p)\le n\log2
 \quad\text{almost surely},\qquad
 \big|\E X_{n+1}^{(\eta)}(p)-\E X_n^{(\eta)}(p)\big|
 \le\log2.
\end{aligned}
 \label{eq:deterministic-bound}
\end{equation}
\end{lemma}

\begin{proof}
For distance, each old edge is replaced by a two-edge network whose
boundary-to-boundary distance is either $1$ or $2$.  Thus
$D_n(p)\le D_{n+1}(p)\le2D_n(p)$.  For resistance, every old unit edge is
replaced by a two-edge network of effective resistance $1/2$ or $2$.
Rayleigh monotonicity and homogeneity compare the new network with the old one
after all edge resistances are multiplied by $1/2$ or by $2$.  Together,
these comparisons prove \eqref{eq:adjacent-Z}, which immediately yields
\eqref{eq:deterministic-bound}.
\end{proof}

\begin{lemma}
\label{lem:first-moment-bounds}
For $\eta\in\{0,1\}$, $p\in[0,1]$, and $n\ge0$,
\begin{equation}
 \E Z_{n+1}^{(\eta)}(p)\ge2p\E Z_n^{(\eta)}(p),
 \qquad
 (2p)^n\le\E Z_n^{(\eta)}(p)\le2^n.
 \label{eq:first-moment-recursion}
\end{equation}
\end{lemma}

\begin{proof}
Taking expectations in \eqref{eq:unified-rde} gives the first inequality.
Its iteration, together with $Z_0^{(\eta)}(p)=1$ and
Lemma~\ref{lem:deterministic-bound}, gives the remaining bounds.
\end{proof}

We next record resistance duality.  For an
environment $\xi=(\xi_u)_{u\in\mathbb T}$, write $R_n(\xi)$ for the
resistance of the resulting depth-$n$ network and put
$\bar\xi=(1-\xi_u)_{u\in\mathbb T}$.  If $\xi$ has parameter $p$, then
$\bar\xi$ has parameter $1-p$; we use these two environments to couple
$R_n(p)$ and $R_n(1-p)$.  Complementing the environment exchanges series
and parallel at every internal node of the construction tree.  Since
$g_-^{(1)}(x,y)=-g_+(-x,-y)$, induction on the coupled construction
tree gives the following.

\begin{lemma}
\label{lem:resistance-duality}
Under this coupling, for every $p\in[0,1]$ and $n\ge0$,
\begin{equation}
 \log R_n(p)=-\log R_n(1-p)
 \qquad\text{almost surely}.
 \label{eq:resistance-duality}
\end{equation}
\end{lemma}

\subsection{Exact resistance CDF operator}
\label{sec:cdf-operator}

In this subsection, we consider only the resistance model and write
$g_-:=g_-^{(1)}$ and $X_n(p):=X_n^{(1)}(p)$.

Let $U$ and $V$ be independent random variables with common CDF $F$.  Denote
the CDFs of $g_+(U,V)$ and $g_-(U,V)$ by $S_F$ and $P_F$, respectively:
\[
 S_F(x):=\Pp\bigl(g_+(U,V)\le x\bigr),
 \qquad
 P_F(x):=\Pp\bigl(g_-(U,V)\le x\bigr).
\]
For $p\in[0,1]$, let $A$ be independent of $(U,V)$, take values in
$\{+,-\}$, and satisfy $\Pp(A=+)=p$ and $\Pp(A=-)=1-p$.  Define
\begin{equation}
 (\mathcal T_pF)(x):=\Pp\bigl(g_A(U,V)\le x\bigr)
 =pS_F(x)+(1-p)P_F(x).
 \label{eq:CDF-operator-definition}
\end{equation}
Thus $\mathcal T_p$ is the one-step CDF update for the resistance recursion.
In particular, if $F$ is the CDF of $X_n(p)$, then $\mathcal T_pF$ is the
CDF of $X_{n+1}(p)$.

We now give an explicit formula for this update when $F$ is absolutely
continuous.  Let $\rho$ be the density of $F$, and recall that
$h(r)=\log(1+e^{-r})$ for $r\ge0$.
Define \[
\begin{aligned}
I_-(F;x)
 &:=\int_0^\infty\int_0^{h(r)}
   \rho(x-s)\rho(x-s-r)\,\dd s\,\dd r,
 \\
 I_+(F;x)
 &:=\int_0^\infty\int_0^{h(r)}
   \rho(x+s)\rho(x+s+r)\,\dd s\,\dd r.
\end{aligned} \]

\begin{lemma}
\label{lem:exact-cdf}
If $F$ is absolutely continuous with density $\rho$, then $S_F$ and $P_F$
satisfy
\begin{equation}
 S_F(x)=F(x)^2-2I_-(F;x),
 \qquad
 P_F(x)=2F(x)-F(x)^2+2I_+(F;x).
 \label{eq:series-parallel-cdf}
\end{equation}
Consequently, if $p=1/2+\delta$, then
\begin{equation}
 \mathcal T_{1/2+\delta}F-F
 =I_+-I_--2\delta F(1-F)-2\delta(I_++I_-).
 \label{eq:exact-cdf-operator}
\end{equation}
\begin{samepage}
Moreover, the operator $\mathcal T_p$ preserves pointwise order and commutes
with translations.  More precisely, for every $p\in[0,1]$ and any CDFs $F$
and $G$, the following properties hold:
\begin{enumerate}
\item[(i)] \emph{Order preservation.}
\[
 F\le G\ \text{pointwise}
 \quad\Longrightarrow\quad
 \mathcal T_pF\le\mathcal T_pG\ \text{pointwise}.
\]
\item[(ii)] \emph{Translation covariance.}  For $a\in\R$, define
$(\tau_aF)(x):=F(x-a)$.  Then
\[
 \mathcal T_p(\tau_aF)=\tau_a(\mathcal T_pF).
\]
\end{enumerate}
\end{samepage}
\end{lemma}

\begin{proof}
By \eqref{eq:log-series-gate},
\[
 g_+(U,V)=(U\vee V)+h(|U-V|).
\]
In particular, $g_+(U,V)\ge U\vee V$.  Since $U$ and $V$ have densities,
the diagonal has probability zero, and
symmetry between $U$ and $V$ gives
\begin{align*}
 S_F(x)
 &=F(x)^2
   -2\Pp\bigl(V\le U\le x,\ U+h(U-V)>x\bigr).
\end{align*}
On the event inside the last probability, make the change of variables
\[
 r=U-V\ge0,
 \qquad
 s=x-U\ge0.
\]
Its Jacobian has absolute value one, and the crossing condition is precisely
$s<h(r)$.  Therefore
\[
\begin{aligned}
 \Pp\bigl(V\le U\le x,\ U+h(U-V)>x\bigr)
 &=\int_0^\infty\int_0^{h(r)}
   \rho(x-s)\rho(x-s-r)\,\dd s\,\dd r\\
 &=I_-(F;x).
\end{aligned}
\]
This proves the formula for $S_F$.

The formula for $P_F$ follows by the same argument, using
\eqref{eq:log-series-gate}.  This proves
\eqref{eq:series-parallel-cdf}.

Substituting these two formulas into
\eqref{eq:CDF-operator-definition}, with $p=1/2+\delta$, yields
\begin{align*}
 {\mathcal T_{1/2+\delta}F-F}
 &{=(\tfrac12+\delta)(F^2-2I_-)
   +(\tfrac12-\delta)(2F-F^2+2I_+)-F}\\
 &{=I_+-I_--2\delta F(1-F)-2\delta(I_++I_-).}
\end{align*}
This is \eqref{eq:exact-cdf-operator}.

It remains to verify the two structural properties of the operator.  Suppose
that $F\le G$ pointwise.  By the standard monotone coupling, there exist
$U_F,V_F\stackrel{\mathrm{i.i.d.}}{\sim}F$ and
$U_G,V_G\stackrel{\mathrm{i.i.d.}}{\sim}G$ on a common probability space such
that $U_F\ge U_G$ and $V_F\ge V_G$ almost surely.  The coordinatewise
monotonicity of $g_+$ and $g_-$ gives
\[
 S_F\le S_G,
 \qquad
 P_F\le P_G.
\]
Hence $\mathcal T_pF\le\mathcal T_pG$, proving preservation of pointwise
order.  Finally, $\tau_aF$ is the CDF of $U+a$.  Translation covariance
\eqref{eq:gate-translation} gives
\[
 \bigl(\mathcal T_p(\tau_aF)\bigr)(x)
 =(\mathcal T_pF)(x-a)
 =\bigl(\tau_a(\mathcal T_pF)\bigr)(x),
\]
so $\mathcal T_p$ commutes with translations.
\end{proof}

\section{Proof of Theorems~\ref{thm:logarithmic-speeds}
and~\ref{thm:first-moment-logarithmic-rates}}
\label{sec:ballistic-proof}

In this section, we prove Theorems~\ref{thm:logarithmic-speeds}
and~\ref{thm:first-moment-logarithmic-rates}.
Section~\ref{sec:first-moment-exponent} establishes the existence of the
first-moment logarithmic rates asserted in
Theorem~\ref{thm:first-moment-logarithmic-rates} via first-moment
submultiplicativity; see Lemma~\ref{lem:first-moment-submultiplicativity}.
Section~\ref{sec:annealed-to-typical} proves the assertions of both theorems
in the supercritical regime $p>1/2$ by controlling the Jensen gap defined in
\eqref{eq:jensen-gap}. Section~\ref{sec:remaining-ranges} completes the
proofs by treating $p\le1/2$.

Throughout this section, fix $\eta\in\{0,1\}$ and $p\in[0,1]$.  For
notational convenience, we suppress these parameters whenever no confusion
can arise and write
\[
 Z_n:=Z_n^{(\eta)}(p),
 \qquad
 X_n:=X_n^{(\eta)}(p).
\]

\subsection{First-moment logarithmic rate}
\label{sec:first-moment-exponent}

\begin{lemma}
\label{lem:first-moment-submultiplicativity}
For each $\eta\in\{0,1\}$ and all $n,r\ge0$,
\begin{equation}
 \E Z_{n+r}\le \E Z_n\,\E Z_r.
 \label{eq:first-moment-submultiplicativity}
\end{equation}
Consequently, the limit
\begin{equation}
 \gamma=\gamma^{(\eta)}(p):=\lim_{n\to\infty}\frac1n\log\E Z_n
 =\inf_{n\ge1}\frac1n\log\E Z_n
 \label{eq:annealed-fekete}
\end{equation}
exists in $(-\infty,\infty)$.
\end{lemma}

For distance, Chen, Derrida, Duquesne, and
Shi~\cite[the argument leading to equation~(1.2)]{ChenDerridaDuquesneShi2026}
established \eqref{eq:first-moment-submultiplicativity} and
\eqref{eq:annealed-fekete}.  We reproduce the proof for completeness.

\begin{proof}
Recall the conditional refinement property in
Proposition~\ref{prop:conditional-refinement}.  Conditionally on
$\mathcal F_n$, each edge of $\mathcal G_n$ is replaced in
$\mathcal G_{n+r}$ by an independent depth-$r$ copy of $\mathcal G_r$.

For distance, fix a shortest path in $\mathcal G_n$, using a deterministic
tie-breaking rule.  In each depth-$r$ copy replacing an edge of this path,
choose a shortest path between its boundary vertices.  Concatenating these
paths gives a path across $\mathcal G_{n+r}$.  Conditional on $\mathcal F_n$,
its expected length is $(\E D_r)D_n$, and hence
\[
 \E[D_{n+r}\mid\mathcal F_n]\le (\E D_r)D_n.
\]
Taking expectations proves \eqref{eq:first-moment-submultiplicativity} for
distance.

For resistance, we use Thomson's principle in its unit-resistance form:
\[
 R_{\mathrm{eff}}(G)
 =\inf_{\vartheta}\sum_{e\in E(G)}\vartheta_e^2,
\]
where the infimum is over unit flows between the boundary vertices of $G$.  Let
$(\theta_e)$ be a minimizing unit flow in $\mathcal G_n$. Such a flow
exists: a minimizing sequence has bounded energy, so its finitely many edge
currents have a convergent subsequence whose limit is again a unit flow. In the
depth-$r$ copy replacing $e$, take a minimizing unit flow
and multiply it by $\theta_e$.
Combining these flows gives a unit flow in $\mathcal G_{n+r}$ with
energy $\sum_e\theta_e^2R_r^{(e)}$.  Therefore,
$R_{n+r}\le\sum_e\theta_e^2R_r^{(e)}$.
Taking conditional expectation gives
\[
 \E[R_{n+r}\mid\mathcal F_n]
 \le (\E R_r)\sum_e\theta_e^2
 = (\E R_r)R_n.
\]
Taking expectations proves \eqref{eq:first-moment-submultiplicativity} for
resistance.

Fekete's subadditive lemma, together with Lemma~\ref{lem:deterministic-bound}, proves
\eqref{eq:annealed-fekete}.
\end{proof}

\subsection{From first-moment logarithmic rate to logarithmic speed}
\label{sec:annealed-to-typical}

Throughout this subsection, we assume that $p>1/2$.  Define the Jensen gap by
\begin{equation}
\label{eq:jensen-gap}
 J_n:=\log\E Z_n-\E X_n\ge0.
\end{equation}
We shall first prove that $J_n$ is uniformly bounded, showing that $\E X_n$
and $\log\E Z_n$ have the same limiting rate after division by $n$, and then
use Lemma~\ref{lem:center-tracking} to control the centered logarithmic
fluctuations.  Together, these estimates yield the $L^1$ convergence in
\eqref{eq:full-range-L1} and the almost-sure convergence.

\begin{proposition}
\label{prop:bounded-Jensen-gap}
For $p>1/2$, the Jensen gaps are uniformly bounded:
\[
 \sup_{n\ge0}J_n<\infty.
\]
Consequently,
\begin{equation}
 \lim_{n\to\infty}\frac1n\E X_n
 =\lim_{n\to\infty}\frac1n\log\E Z_n
 =\gamma.
 \label{eq:mean-log-limit}
\end{equation}
\end{proposition}

We defer the proof of Proposition~\ref{prop:bounded-Jensen-gap}.  Assuming it,
we next prove Theorem~\ref{thm:logarithmic-speeds} for $p>1/2$.  We begin with
a preliminary lemma.

\begin{lemma}
\label{lem:center-tracking}
For $p>1/2$, let $X_n'$ be an independent copy of $X_n$.  With
$C:=\log2/(p-1/2)$, we have
\begin{equation}
 \sup_{n\ge0}\E|X_n-X_n'|\le C,
 \qquad
 \sup_{n\ge0}\E|X_n-\E X_n|\le C.
 \label{eq:uniform-width}
\end{equation}
Consequently,
\begin{equation}
 \frac{X_n-\E X_n}{n}\longrightarrow0
 \qquad\text{a.s. and in }L^1.
\label{eq:center-tracking}
\end{equation}
\end{lemma}

\begin{proof}
Since $p>1/2$, the coefficient of the $h$ term in the drift identity
\eqref{eq:drift-identity} is nonnegative: it equals $p$ for distance and
$2p-1$ for resistance.  Combining that identity with the adjacent-generation
bound in Lemma~\ref{lem:deterministic-bound}, we obtain
\[
 0\le\left(p-\frac12\right)\E|X_n-X_n'|
 \le \E X_{n+1}-\E X_n\le\log2,
\]
which proves the first bound in \eqref{eq:uniform-width}.  Conditional Jensen
gives
\[
 \E|X_n-\E X_n|
 =\E\left|X_n-\E[X_n'\mid X_n]\right|
 \le\E|X_n-X_n'|,
\]
which proves the second bound.  The $L^1$ convergence in
\eqref{eq:center-tracking} follows immediately.

For the almost-sure statement, Markov's inequality and
\eqref{eq:uniform-width} give, for every $j\ge1$,
\[
 \sum_{k\ge1}\Pp\left\{
  |X_{k^2}-\E X_{k^2}|>\frac{k^2}{j}
 \right\}
 \le jC\sum_{k\ge1}\frac1{k^2}<\infty.
\]
Thus Borel--Cantelli yields
\[
 \frac{X_{k^2}-\E X_{k^2}}{k^2}\longrightarrow0
 \qquad\text{almost surely}.
\]
If $k^2\le n\le(k+1)^2$, Lemma~\ref{lem:deterministic-bound} gives
\[
 \big|(X_n-\E X_n)-(X_{k^2}-\E X_{k^2})\big|
 \le2(2k+1)\log2.
\]
Dividing by $n$ completes the proof.
\end{proof}

\begin{proof}[Proof of Theorems~\ref{thm:logarithmic-speeds}
and~\ref{thm:first-moment-logarithmic-rates} for $p>1/2$]
Section~\ref{sec:first-moment-exponent} establishes the existence of the
first-moment logarithmic rates $\gamma^{(0)}(p)=\gamma_D(p)$ and
$\gamma^{(1)}(p)=\gamma_R(p)$.
Proposition~\ref{prop:bounded-Jensen-gap} and
Lemma~\ref{lem:center-tracking} then imply
\begin{equation}
 \frac{X_n}{n}\longrightarrow \gamma=\gamma^{(\eta)}(p)
 \qquad\text{a.s. and in }L^1.
\label{eq:fixed-supercritical-convergence}
\end{equation}
Setting $\eta=0$ and $\eta=1$ in
\eqref{eq:fixed-supercritical-convergence} proves that
$\gamma_D(p)=v_D(p)$ and $\gamma_R(p)=v_R(p)$, respectively, for $p>1/2$,
which proves
\eqref{eq:supercritical-resistance-identification} in this range.  Finally,
\eqref{eq:first-moment-recursion} and
\eqref{eq:annealed-fekete} prove
\eqref{eq:supercritical-speed-bounds}.
\end{proof}

We now prove Proposition~\ref{prop:bounded-Jensen-gap}.  The idea is as follows.
We first use the
second-moment method to show that $Z_n\ge\E Z_n/2$ with uniformly positive
probability.  On this
event, we have $J_n\le X_n-\E X_n+\log2$; since its probability is bounded away from
zero, the second bound in \eqref{eq:uniform-width} yields the desired uniform
bound on $J_n$.

We begin by bounding the second-moment ratio.  Set
\[
 K_n:=\frac{\E Z_n^2}{(\E Z_n)^2}.
\]

\begin{lemma}
\label{lem:normalized-L2}
For $p>1/2$,
\begin{equation*}
\sup_{n\ge0}K_n
 \le K
 :=\frac{2p+(1-p)4^{-\eta}}{2p(2p-1)}.
\end{equation*}
\end{lemma}

\begin{proof}
Let $Z_n'$ be an independent copy of $Z_n$.  We have
\[
 \Pi_0(Z_n,Z_n')^2\le Z_nZ_n',
 \qquad
 \Pi_1(Z_n,Z_n')^2\le\frac{Z_nZ_n'}{4}.
\]
Using independence and
$\E(Z_n+Z_n')^2=2\E Z_n^2+2(\E Z_n)^2$, we obtain
\begin{align*}
    \E Z_{n+1}^2=p\E(Z_n+Z_n')^2+(1-p)\E\Pi_\eta(Z_n,Z_n')^2
    \le 2p\E Z_n^2+\{2p+(1-p)4^{-\eta}\}(\E Z_n)^2.
\end{align*}
After division by $(\E Z_{n+1})^2$ and using
\eqref{eq:first-moment-recursion},
\begin{equation}
 K_{n+1}
 \le\frac1{2p}K_n
 +\frac{2p+(1-p)4^{-\eta}}{4p^2}.
 \label{eq:normalized-L2-recursion}
\end{equation}
The affine map on the right-hand side of
\eqref{eq:normalized-L2-recursion} has slope $1/(2p)<1$ and fixed point $K$.
Since $K_0=1\le K$, induction gives $K_n\le K$ for every $n\ge0$.
\end{proof}

\begin{proof}[Proof of Proposition~\ref{prop:bounded-Jensen-gap}]
By Lemma~\ref{lem:normalized-L2} and the Paley--Zygmund inequality at
threshold $1/2$,
\begin{equation}
 \Pp\left\{Z_n\ge\frac{\E Z_n}{2}\right\}\ge\frac1{4K}.
 \label{eq:Paley-Zygmund}
\end{equation}
Observe that on the event in \eqref{eq:Paley-Zygmund},
\begin{equation}
 X_n-\E X_n\ge\log(\E Z_n/2)-\E X_n=J_n-\log2.
 \label{eq:jensen-event-comparison}
\end{equation}
Consequently, by \eqref{eq:Paley-Zygmund},
\[
\begin{aligned}
 \frac{(J_n-\log2)_+}{4K}
 &\le\E\left[(J_n-\log2)_+
   \ind_{\{Z_n\ge\E Z_n/2\}}\right]
 \\ &\le\E\left[|X_n-\E X_n|
   \ind_{\{Z_n\ge\E Z_n/2\}}\right]
 \le\E|X_n-\E X_n|\le C.
\end{aligned}
\]
Therefore,
\begin{equation}
 0\le J_n\le\log2+4KC.
 \label{eq:Jensen-gap-bound}
\end{equation}

Combining Lemma~\ref{lem:first-moment-submultiplicativity} with
\eqref{eq:Jensen-gap-bound},
\[
 \frac{\E X_n}{n}
 =\frac{\log\E Z_n}{n}-\frac{J_n}{n}
 \longrightarrow \gamma,
\]
which proves \eqref{eq:mean-log-limit} and completes the proof.
\end{proof}

\subsection{Remaining parameter ranges}
\label{sec:remaining-ranges}

\begin{proof}[Proof of Theorem~\ref{thm:logarithmic-speeds}
for $p\le 1/2$]
For distance, Chen, Derrida, Duquesne, and
Shi~\cite{ChenDerridaDuquesneShi2026} showed that, for $p\le1/2$, its
first-moment logarithmic rate vanishes, that is,
\begin{equation}
 \lim_{n\to\infty}\frac1n\log\E D_n(p)=0.
 \label{eq:distance-nonpositive-annealed-speed}
\end{equation}
Since $D_n(p)\ge1$,
Jensen's inequality gives
\[
 0\le \frac1n\E\log D_n(p)
 \le \frac1n\log\E D_n(p)\longrightarrow0.
\]
Thus $n^{-1}\log D_n(p)\to0$ in $L^1$.
The convergence also holds almost surely.  Indeed, for every
$\varepsilon>0$,
\eqref{eq:distance-nonpositive-annealed-speed} gives, for all sufficiently
large $n$,
\[
 \Pp\{\log D_n(p)>\varepsilon n\}
 \le e^{-\varepsilon n}\E D_n(p)
 \le e^{-\varepsilon n/2}.
\]
Borel--Cantelli and $D_n(p)\ge1$ prove that
$n^{-1}\log D_n(p)\to0$ almost surely.

For $p<1/2$, define $v_R(p)=-v_R(1-p)$.  Since $1-p>1/2$,
\eqref{eq:fixed-supercritical-convergence} at parameter $1-p$,
together with Lemma~\ref{lem:resistance-duality}, gives
\[
 \frac1n\log R_n(p)\longrightarrow-v_R(1-p)=v_R(p)
\qquad\text{a.s. and in }L^1.
\]

At $p=1/2$, let $\xi$ be the parameter-$1/2$ environment and let $\bar\xi$ be its
complement.  Thomson's principle applied to a unit flow along a shortest path,
together with \eqref{eq:resistance-duality}, gives
\[
 -\log D_n(\bar\xi)
 \le \log R_n(\xi)
 \le \log D_n(\xi).
\]
Since $D_n(\xi),D_n(\bar\xi)\ge1$, the preceding sandwich implies
\[
 |\log R_n(\xi)|
 \le \log D_n(\xi)+\log D_n(\bar\xi).
\]
The $L^1$ and almost-sure distance result proved above applies to both $\xi$
and $\bar\xi$.  Dividing by $n$ and taking expectations proves the $L^1$
convergence; the almost-sure convergence follows from the same inequality.
\end{proof}

\begin{proof}[Proof of Theorem~\ref{thm:first-moment-logarithmic-rates}
for $p\le 1/2$]
For distance, \eqref{eq:distance-nonpositive-annealed-speed} and
Theorem~\ref{thm:logarithmic-speeds} give
$\gamma_D(p)=v_D(p)=0$ for $p\le1/2$.

For resistance, $R_n(0)=2^{-n}$, so
$\gamma_R(0)=v_R(0)=-\log2$.  At $p=1/2$,
Jensen's inequality, the $L^1$ convergence established above,
the bound $R_n(1/2)\le D_n(1/2)$, and
\eqref{eq:distance-nonpositive-annealed-speed} yield
\[
\begin{aligned}
 0
 &=\lim_{n\to\infty}\frac1n\E\log R_n(1/2)
 \le\liminf_{n\to\infty}\frac1n\log\E R_n(1/2)
 \\ &\le\limsup_{n\to\infty}\frac1n\log\E R_n(1/2)
 \le\lim_{n\to\infty}\frac1n\log\E D_n(1/2)=0.
\end{aligned}
\]
Thus $\gamma_R(1/2)=v_R(1/2)=0$, which proves
\eqref{eq:supercritical-resistance-identification} at $p=1/2$.

For $0<p<1/2$, Jensen's inequality and the $L^1$ convergence in
Theorem~\ref{thm:logarithmic-speeds} give $\gamma_R(p)\ge v_R(p)$.
If $\gamma_R(p)=v_R(p)$, then Lemma~\ref{lem:first-moment-bounds} gives
$\gamma_R(p)\ge\log(2p)$, and the desired identity follows.  It therefore
suffices to prove that if $\gamma_R(p)>v_R(p)$, then
$\gamma_R(p)=\log(2p)$.

Suppose then that $\gamma_R(p)>v_R(p)$ and set
$m_n=\E R_n(p)$ and
\[
 \widehat R_n=\frac{R_n(p)}{m_n}.
\]
Then $\E\widehat R_n=1$ and
\[
 \frac1n\log\widehat R_n
 =\frac1n\log R_n(p)-\frac1n\log m_n
 \longrightarrow v_R(p)-\gamma_R(p)<0
 \qquad\text{a.s.}
\]
Hence $\widehat R_n\to0$ a.s.  If $\widehat R_n'$ is an independent copy of
$\widehat R_n$, then, for every $\tau>0$,
\[
 c_n:=\E\frac{\widehat R_n\widehat R_n'}
                    {\widehat R_n+\widehat R_n'}
 \le \E(\widehat R_n\wedge\widehat R_n')
 \le \tau+\E\bigl[\widehat R_n
                   \ind_{\{\widehat R_n'>\tau\}}\bigr]
 =\tau+\Pp(\widehat R_n>\tau).
\]
First let $n\to\infty$ with $\tau>0$ fixed, and then let $\tau\downarrow0$.
Consequently $c_n\to0$, and the exact first-moment recursion yields
\[
\begin{aligned}
 \frac{m_{n+1}}{m_n}
 &=2p+(1-p)c_n\longrightarrow2p,\\
 \gamma_R(p)
 &=\lim_{n\to\infty}\frac1n
   \sum_{j=0}^{n-1}\log\frac{m_{j+1}}{m_j}
 =\log(2p).
\end{aligned}
\]
This completes the proof.
\end{proof}

\section{Proof of Theorem~\ref{thm:near-critical-speed}}
\label{sec:two-sided-barriers}

This section is devoted to the proof of
Theorem~\ref{thm:near-critical-speed}. The following proposition establishes
the existence of the smallest admissible parameter
$\lambda_*\in(0,\infty)$ for the boundary-value problem
\eqref{eq:main-W-bvp}. The existence and uniqueness classification is a
special case of \cite[Proposition~2]{EnguicaGavioliSanchez2013}.
A self-contained proof, including the explicit bounds below, is given in
Appendix~\ref{app:ode}.
\begin{proposition}
\label{prop:Cstar-halfline}
For $\lambda>0$, consider the boundary-value problem
\begin{equation}\label{eqn:ode-W}
\begin{aligned}
 &W^2W'-\lambda W+u(1-u)=0,\qquad 0<u<1,\\
 &W(u)>0\quad\text{for }\ 0<u<1, \qquad
 W(0)=W(1)=0.
\end{aligned}
\end{equation}
Recall that $\lambda>0$ is \emph{admissible} if \eqref{eqn:ode-W} admits a solution
$W\in C([0,1])\cap C^1((0,1))$.
Then there exists $\lambda_*\in(0,\infty)$ such that $\lambda$ is admissible
if and only if $\lambda\ge\lambda_*$. Moreover, we have
\begin{equation}\label{eq:Z-halfline}
    \left(\frac{4}{15\sqrt2}\right)^{2/3} \le \lambda_* \le \frac{\sqrt{3}}{2}.
\end{equation}
\end{proposition}
\begin{remark}
  The upper bound $\frac{\sqrt3}{2}$ above is an elementary bound used only to prove the finiteness of $\lambda_*$.
 The estimate in
\cite[Proposition~2]{EnguicaGavioliSanchez2013} gives
\( \lambda_*\le2^{-1/3}\), and the quantitative lower bound in
\cite[Lemma~3.1(a)]{EnguicaGavioliSanchez2013} is zero.
The explicit positive lower bound in \eqref{eq:Z-halfline} follows from
the integral estimate in Appendix~\ref{app:ode}.
\end{remark}
Therefore, in Section~\ref{sec:two-sided-barriers} we focus on proving the asymptotics of $v_R$ and $\gamma_R$.
We first collect the constants and
functions used in the proof and then prove the upper and lower bounds for $v_R$.
Finally, we derive the asymptotics of $\gamma_R$ using Theorem~\ref{thm:first-moment-logarithmic-rates}.

From this section onward, we specialize to effective resistance and write
$v(p):=v_R(p)$.

\subsection{Preliminaries}

The following lemma calculates the coefficient appearing in the later
estimates in Lemmas~\ref{lem:weighted-consistency}
and~\ref{lem:hard-edge-consistency}.
Recall that \(h(r)=\log(1+e^{-r})\) for \(r\ge0\).

\begin{lemma}
\label{lem:diffusion-coefficient}
The following equality holds:
\begin{equation*}
a:=2\int_0^\infty\bigl(h(r)^2+rh(r)\bigr)\,\dd r=2\zeta(3).
\end{equation*}

\end{lemma}
\begin{remark}
   Morfe~\cite[Appendix~C]{Morfe2026} denotes this coefficient by \(a_R\)
and computes \(a_R=2\zeta(3)\). We include the short calculation here
because our proof uses this precise normalization.
\end{remark}
\begin{proof}

With $u=(1+e^r)^{-1}$, symmetry under $u\mapsto1-u$ gives
\[
\begin{aligned}
\int_0^\infty\bigl(h(r)^2+rh(r)\bigr)\,\dd r
 &=\int_0^{1/2}
   \frac{\log u\log(1-u)}{u(1-u)}\,\dd u
 =\int_0^1\frac{\log u\log(1-u)}u\,\dd u\\
 &=\sum_{n\ge1}\frac1n
   \int_0^1u^{n-1}(-\log u)\,\dd u
 =\sum_{n\ge1}\frac1{n^3}
 =\zeta(3),
\end{aligned}
\]
where the first line also uses
$1/[u(1-u)]=1/u+1/(1-u)$, and the second follows from
$-\log(1-u)=\sum_{n\ge1}u^n/n$ and Tonelli's theorem.
\end{proof}

For later use set
\begin{equation}
 \beta=\left(\frac2a\right)^{1/3},
 \qquad
 \kappa_{\lambda}=\frac{2\lambda}{\beta}=2\zeta(3)^{1/3}\lambda.
 \label{eq:beta-kappa}
\end{equation}

For each $\lambda>0$, let
$W_\lambda\in C([0,1])\cap C^1((0,1))$ be the unique function satisfying
\[
W_\lambda(u)^2W_\lambda'(u)-\lambda W_\lambda(u)+u(1-u)=0,
\qquad 0<u<1,
\]
with
\[
W_\lambda(1)=0,
\qquad
W_\lambda(u)>0
\quad\text{for }0<u<1.
\]
Its existence, uniqueness, and endpoint behavior are established in
Appendix~\ref{app:ode}.

The following two propositions, proved in Appendix~\ref{app:regularity},
provide the distribution functions and their properties needed for the upper
and lower bounds, respectively.

\begin{proposition}
\label{prop:full-line-distribution}
For every $\lambda>\lambda_*$, there is a unique increasing
$\Phi_\lambda\in C^\infty(\mathbb R;(0,1))$ such that
\begin{equation}
\begin{aligned}
 \Phi_\lambda'(z)&=\beta W_\lambda(1-\Phi_\lambda(z)),\\
 \Phi_\lambda(0)&=\frac12,\qquad
 \lim_{z\to-\infty}\Phi_\lambda(z)=0, \qquad
 \lim_{z\to\infty}\Phi_\lambda(z)=1.
\end{aligned}
 \label{eq:Phi-phase-definition}
\end{equation}
Put $q_\lambda=\Phi_\lambda'$.  Then $q_\lambda>0$ and
\begin{equation}
 a q_\lambda q_\lambda'
 -2\Phi_\lambda(1-\Phi_\lambda)
 =-\kappa_\lambda q_\lambda
 \quad\text{on }\mathbb R,
 \label{eq:upper-profile-equation}
\end{equation}
and there exists $M_\lambda<\infty$ such that
\begin{equation}
 \max_{1\le j\le3}|q_\lambda^{(j)}(z)|
 \le M_\lambda q_\lambda(z),
 \qquad z\in\mathbb R.
 \label{eq:full-line-relative-bounds}
\end{equation}
Moreover,
\begin{equation}
 \frac{q_\lambda(z)}{\Phi_\lambda(z)}\longrightarrow\frac\beta\lambda
 \quad(z\to-\infty),
 \qquad
 \frac{q_\lambda(z)}{1-\Phi_\lambda(z)}\longrightarrow\frac\beta\lambda
 \quad(z\to\infty).
 \label{eq:full-line-tail-ratios}
\end{equation}
Consequently, there exist constants $c_\lambda,C_\lambda>0$ such that
\[
\begin{aligned}
 \Phi_\lambda(z)\le C_\lambda e^{c_\lambda z}
 \quad(z\le0),\qquad
 1-\Phi_\lambda(z)\le C_\lambda e^{-c_\lambda z}
 \quad(z\ge0),\qquad
 \int_{\mathbb R}|x|\,\dd\Phi_\lambda(x)&<\infty.
\end{aligned}
\]
\end{proposition}

\begin{proposition}
\label{prop:hard-edge-distribution}
For every $0<\lambda<\lambda_*$, there is a unique increasing
$\Psi_\lambda\in C^\infty([0,\infty);[0,1))$ such that
\begin{equation}
\begin{aligned}
 \Psi_\lambda'(z)&=\beta W_\lambda(\Psi_\lambda(z)),\\
 \Psi_\lambda(0)&=0,\qquad
 \Psi_\lambda(z)\longrightarrow1
 \quad(z\to\infty).
\end{aligned}
 \label{eq:Psi-subcritical-definition}
\end{equation}
When extended by zero to $(-\infty,0)$, $\Psi_{\lambda}$ is an absolutely continuous CDF supported on
$[0,\infty)$.

Put $q_\lambda=\Psi_\lambda'$ on $[0,\infty)$. At the endpoint $0$,
all derivatives are taken from the right.  Then
\begin{equation}
 q_\lambda(0)=\beta W_\lambda(0)>0,
 \qquad
 a q_\lambda q_\lambda'
 +2\Psi_\lambda(1-\Psi_\lambda)
 =\kappa_\lambda q_\lambda
 \quad(z\ge0),
 \label{eq:lower-profile-equation}
\end{equation}
and there exists $M_\lambda<\infty$ such that
\begin{equation}
 \max_{1\le j\le3}|q_\lambda^{(j)}(z)|
 \le M_\lambda q_\lambda(z),
 \qquad z\ge0.
 \label{eq:subcritical-relative-bounds}
\end{equation}

Furthermore, $q_\lambda$ has a strictly positive extension
$\overline q_\lambda\in C^3(\mathbb R)$, equal to $q_\lambda$ on
$[0,\infty)$, such that, for some $\overline M_\lambda<\infty$,
\begin{equation}
 \max_{1\le j\le3}|\overline q_\lambda^{(j)}(z)|
 \le \overline M_\lambda\overline q_\lambda(z),
 \qquad z\in\mathbb R.
 \label{eq:q-extension-bounds}
\end{equation}
Moreover,
\begin{equation}
 \frac{q_\lambda(z)}{1-\Psi_\lambda(z)}
 \longrightarrow\frac\beta\lambda
 \quad(z\to\infty).
 \label{eq:subcritical-right-tail}
\end{equation}
Consequently, there exist constants $c_\lambda,C_\lambda>0$ such that
\[
\begin{aligned}
 1-\Psi_\lambda(z)\le C_\lambda e^{-c_\lambda z}
 \quad(z\ge0),\qquad
 \int_{[0,\infty)}x\,\dd\Psi_\lambda(x)&<\infty.
\end{aligned}
\]

\end{proposition}

Throughout this section $\varepsilon>0$ is small and
\[
 p_+=\frac12+\varepsilon^3,
 \qquad
 p_-=\frac12-\varepsilon^3.
\]
Here and below, the bias means $p-\frac{1}{2}$; thus $p_+$ has positive
bias and $p_-$ has negative bias.  Let $F_n^\pm$ denote the CDFs of
$X_n(p_\pm)$, respectively.

\subsection{Upper bound}

Fix $\lambda>\lambda_*$ and $\kappa>\kappa_{\lambda}$.  It is enough to construct an
integrable $Y_\varepsilon$ with CDF $G_\varepsilon$ satisfying
\begin{align}
 &F_0\ge G_\varepsilon, \label{upp_barrier_init}\\
 &\mathcal T_{p_+}G_\varepsilon(x)
 \ge G_\varepsilon(x-\kappa\varepsilon^2),
 \qquad x\in\mathbb R,\label{upp_barrier_ineq}
\end{align}
where $F_0$ is the CDF of $X_0=0$.

Indeed, if such a CDF $G_\varepsilon$ exists, then order preservation and
translation covariance in Lemma~\ref{lem:exact-cdf} give inductively
\[
 F_n^+(x)\ge G_\varepsilon(x-n\kappa\varepsilon^2).
\]
We prove this by induction: If the inequality holds at time $n$, then
\[
\begin{aligned}
 F_{n+1}^+
 =\mathcal T_{p_+}F_n^+
 \ge \mathcal T_{p_+}
     \bigl(\tau_{n\kappa\varepsilon^2}G_\varepsilon\bigr)
 =\tau_{n\kappa\varepsilon^2}
     \bigl(\mathcal T_{p_+}G_\varepsilon\bigr)
 \ge\tau_{(n+1)\kappa\varepsilon^2}G_\varepsilon.
\end{aligned}
\]
Consequently,
\[
 X_n(p_+)\stle n\kappa\varepsilon^2+Y_\varepsilon.
\]
Since $X_n(p_+)$ and $Y_\varepsilon$ are integrable, taking expectations,
dividing by $n$, and using Theorem~\ref{thm:logarithmic-speeds} gives
\[
 \frac{\E X_n(p_+)}n
 \le \kappa\varepsilon^2+\frac{\E Y_\varepsilon}n
 \longrightarrow \kappa\varepsilon^2.
\]
Thus
\begin{equation}
 v(p_+)\le\kappa\varepsilon^2.
 \label{eq:upper-speed-fixed-lambda}
\end{equation}
Taking $\delta=\varepsilon^3$ in
\eqref{eq:upper-speed-fixed-lambda} yields
\[
 \limsup_{\delta\downarrow0}
 \frac{v(\frac12+\delta)}{\delta^{2/3}}
 \le\kappa.
\]
Letting first $\kappa\downarrow\kappa_{\lambda}$ and then
$\lambda\downarrow\lambda_*$, and using \eqref{eq:beta-kappa}, gives the upper bound
\begin{equation}
 \limsup_{\delta\downarrow0}
 \frac{v(\frac12+\delta)}{\delta^{2/3}}
 \le2\zeta(3)^{1/3}\lambda_*.
 \label{eq:sharp-upper-bound}
\end{equation}

It remains to construct an integrable $Y_\varepsilon$ whose CDF
$G_\varepsilon$ has the two properties in \eqref{upp_barrier_init}
and \eqref{upp_barrier_ineq}.

\paragraph*{Step 1: Construct a candidate satisfying \eqref{upp_barrier_ineq}.}
 Let
$\Phi=\Phi_{\lambda}$ and $q=\Phi'$ be as in
Proposition~\ref{prop:full-line-distribution}.  For $z_0\in\mathbb R$
and $\varepsilon>0$, set
\[
 F_{\varepsilon,z_0}(x)=\Phi(z_0+\varepsilon x),
 \qquad
 \xi=z_0+\varepsilon x.
\]

The following lemma gives the
required estimates for $\Phi$. Morfe~\cite[Proposition~10]{Morfe2026} obtains the same leading order $a\varepsilon^3q(\xi)q'(\xi)$ for the resistance recursion, with a  uniform $o(\varepsilon^3)$ error.  For the function used here, the lemma below strengthens the error bound to $O(\varepsilon^5q)$.
\begin{lemma}
\label{lem:weighted-consistency}
There are constants $C_\lambda<\infty$ and $\varepsilon_\lambda>0$
such that, for all $x,z_0\in\mathbb R$ and
$0<\varepsilon<\varepsilon_\lambda$,
\begin{align}
 \left|I_+(F_{\varepsilon,z_0};x)-I_-(F_{\varepsilon,z_0};x)
       -a\varepsilon^3q(\xi)q'(\xi)\right|
 &\le C_\lambda\varepsilon^5q(\xi),
 \label{eq:weighted-difference}\\
 I_+(F_{\varepsilon,z_0};x)+I_-(F_{\varepsilon,z_0};x)
 &\le C_\lambda\varepsilon^2q(\xi).
 \label{eq:weighted-sum}
\end{align}
\end{lemma}

Combining Lemmas~\ref{lem:exact-cdf} and
\ref{lem:weighted-consistency} with
\eqref{eq:upper-profile-equation} gives
\begin{equation*}
 \mathcal T_{p_+}F_{\varepsilon,z_0}(x)
 -F_{\varepsilon,z_0}(x)
 =-\kappa_{\lambda}\varepsilon^3q(\xi)
  +O_{\lambda}(\varepsilon^5q(\xi)),
\end{equation*}
for all $x,z_0\in\R$.  Taylor expansion gives
\begin{equation*}
 F_{\varepsilon,z_0}(x-\kappa\varepsilon^2)
 -F_{\varepsilon,z_0}(x)
 =-\kappa\varepsilon^3q(\xi)
  +O_{\lambda,\kappa}(\varepsilon^6q(\xi)).
\end{equation*}
Hence, for all sufficiently small $\varepsilon$ and all
$x,z_0\in\mathbb R$,
\begin{equation}
 \mathcal T_{p_+}F_{\varepsilon,z_0}(x)
 \ge F_{\varepsilon,z_0}(x-\kappa\varepsilon^2)
 +\frac{\kappa-\kappa_{\lambda}}{2}\varepsilon^3q(\xi).
 \label{eq:strict-untruncated-upper}
\end{equation}

Thus $F_{\varepsilon,z_0}$ satisfies the required inequality \eqref{upp_barrier_ineq}, with a
strict margin.  However, it does not satisfy \eqref{upp_barrier_init}, since
$F_{\varepsilon,z_0}(x)>0=F_0(x)$ for $x<0$.

\paragraph*{Step 2: Use truncation to construct $G_\varepsilon$.}
Set
\begin{equation*}
 \alpha_\varepsilon=\varepsilon^4,
 \qquad
 \Phi(z_\varepsilon)=\alpha_\varepsilon,
 \qquad
 F_\varepsilon(x)=\Phi(z_\varepsilon+\varepsilon x).
\end{equation*}
Let $\widetilde Y_\varepsilon$ have CDF $F_\varepsilon$, put $Y_\varepsilon=\widetilde Y_\varepsilon\vee0$, and denote its
CDF by
\begin{equation*}
 G_\varepsilon(x)=
 \begin{cases}
 0,&x<0,\\
 F_\varepsilon(x),&x\ge0,
 \end{cases}
 \qquad G_\varepsilon(0)=\alpha_\varepsilon.
\end{equation*}
Then $F_0\ge G_\varepsilon$.  The truncation error is controlled by the
following lemma.

\begin{lemma}
\label{lem:upper-cutoff-error}
For all sufficiently small $\varepsilon$ and all
$x\ge\kappa\varepsilon^2$,
\begin{equation}
 0\le
 \mathcal T_{p_+}F_\varepsilon(x)
 -\mathcal T_{p_+}G_\varepsilon(x)
 \le C_{\lambda}\varepsilon^4q(z_\varepsilon+\varepsilon x).
 \label{eq:upper-operator-cutoff-error}
\end{equation}
\end{lemma}

For all sufficiently small $\varepsilon$ and
$x\ge\kappa\varepsilon^2$, Lemma~\ref{lem:upper-cutoff-error} and
\eqref{eq:strict-untruncated-upper} give
\[
\begin{aligned}
 \mathcal T_{p_+}G_\varepsilon(x)
 &\ge F_\varepsilon(x-\kappa\varepsilon^2)+\left(
     \frac{\kappa-\kappa_{\lambda}}2\varepsilon^3
     -C_{\lambda}\varepsilon^4
   \right)q(z_\varepsilon+\varepsilon x)\\
 &\ge G_\varepsilon(x-\kappa\varepsilon^2).
\end{aligned}
\]
Here the last inequality uses the definition of $G_\varepsilon$:
\(
 F_\varepsilon(x-\kappa\varepsilon^2)
 =G_\varepsilon(x-\kappa\varepsilon^2)
\)
for $x-\kappa\varepsilon^2\ge0.$

For $x<\kappa\varepsilon^2$,
$G_\varepsilon(x-\kappa\varepsilon^2)=0$, while
$\mathcal T_{p_+}G_\varepsilon(x)\ge0$ because it is a CDF.  Thus the two
cases give the global inequality
\begin{equation*}
 \mathcal T_{p_+}G_\varepsilon(x)
 \ge G_\varepsilon(x-\kappa\varepsilon^2),
 \qquad x\in\mathbb R.
\end{equation*}
Moreover, Proposition~\ref{prop:full-line-distribution} implies that the law
with CDF $\Phi$ has a finite absolute first moment.  Hence $\widetilde Y_\varepsilon$, and therefore
$Y_\varepsilon$, is integrable.  Thus $G_\varepsilon$ has the two properties
used above.

We now prove the two estimates used in the construction.

\begin{proof}[Proof of Lemma~\ref{lem:weighted-consistency}]

The density of $F_{\varepsilon,z_0}$ is
$\rho(x)=\varepsilon q(z_0+\varepsilon x)$.
For fixed $r\ge0$ and
$0\le s\le h(r)$, set
\[
 Q_{s,r}(t)=q(\xi+ts)q(\xi+t(s+r)).
\]
Then
\[
 I_+(F_{\varepsilon,z_0};x)-I_-(F_{\varepsilon,z_0};x)
 =\varepsilon^2\int_0^\infty\int_0^{h(r)}
 \bigl(Q_{s,r}(\varepsilon)-Q_{s,r}(-\varepsilon)\bigr)
 \,\dd s\,\dd r.
\]
The Taylor expansion gives
\[
 Q_{s,r}(\varepsilon)-Q_{s,r}(-\varepsilon)
 =2\varepsilon(2s+r)q(\xi)q'(\xi)+R_{s,r},
\]
where
\[
 |R_{s,r}|\le\frac{\varepsilon^3}{3}
 \sup_{|t|\le\varepsilon}|Q_{s,r}'''(t)|.
\]

To bound the remainder, we first derive two estimates on $q$. Write $M:=M_\lambda$, and let $C_\lambda$ denote constants that may change from line to line. Since $q>0$, the case $j=1$ in
\eqref{eq:full-line-relative-bounds} gives
\begin{equation}
 e^{-M|y|}q(z)\le q(z+y)\le e^{M|y|}q(z),
 \qquad z,y\in\mathbb R.
 \label{eq:q-log-lipschitz}
\end{equation}
The monotonicity and range of $\Phi$ imply
$\int_{\mathbb R}q\le1$.  Hence
\begin{equation}
 1\ge\int_{z-1}^{z+1}q(u)\,\dd u
 \ge 2e^{-M}q(z),
 \qquad
 \sup_{z\in\mathbb R}q(z)\le\frac{e^M}{2}.
 \label{eq:q-uniform-bound}
\end{equation}
Thus every occurrence of $q(z)^2$ below is bounded by a constant times
$q(z)$, uniformly in $z$.

By Leibniz's rule and
\eqref{eq:full-line-relative-bounds}, each term in
$Q_{s,r}'''(t)$ is bounded, up to a constant depending only on $M$, by
$s^j(s+r)^{3-j}q(\xi+ts)q(\xi+t(s+r))$ for some $0\le j\le3$.
Since $0\le s\le h(r)\le\log2$,
\eqref{eq:q-log-lipschitz} therefore gives
\[
\begin{aligned}
 \sup_{|t|\le\varepsilon}|Q_{s,r}'''(t)|
 &\le C_\lambda(2s+r)^3
   \sup_{|t|\le\varepsilon}
   q(\xi+ts)q(\xi+t(s+r))\\
 &\le C_\lambda(1+r)^3e^{2M\varepsilon(1+r)}q(\xi)^2.
\end{aligned}
\]

Choose $\varepsilon_\lambda>0$ so that
$2M\varepsilon_\lambda\le1/2$. For
$0<\varepsilon\le\varepsilon_\lambda$, the contribution of $R_{s,r}$ to
$I_+-I_-$ is bounded in absolute value by
\[
 \begin{aligned}
 \varepsilon^2
 \int_0^\infty\int_0^{h(r)}
 |R_{s,r}|\,\dd s\,\dd r
 &\le C_\lambda\varepsilon^5q(\xi)^2
   \int_0^\infty h(r)(1+r)^3
   e^{2M\varepsilon(1+r)}\,\dd r\\
 &\le C_\lambda\varepsilon^5q(\xi)^2
   \int_0^\infty(1+r)^3e^{-r/2}\,\dd r\\
 &\le C_\lambda\varepsilon^5q(\xi).
 \end{aligned}
\]
Here the second inequality uses $h(r)\le e^{-r}$, with the constant
$e^{1/2}$ absorbed into $C_\lambda$, and the last one follows from
\eqref{eq:q-uniform-bound}.

The linear term equals
\[
 2\varepsilon^3q(\xi)q'(\xi)
 \int_0^\infty\int_0^{h(r)}(2s+r)\,\dd s\,\dd r
 =a\varepsilon^3q(\xi)q'(\xi)
\]
by Lemma~\ref{lem:diffusion-coefficient}.  Together with the preceding
remainder estimate, this proves
\eqref{eq:weighted-difference}.

The same relative estimate applied
to the sum of the two products, now with the prefactor $\varepsilon^2$, gives
\[
 I_++I_-
 \le C_\lambda\varepsilon^2q(\xi)^2
 \int_0^\infty e^{-r/2}\,\dd r
 \le C_\lambda\varepsilon^2q(\xi),
\]
and proves \eqref{eq:weighted-sum}.
\end{proof}

\begin{proof}[Proof of Lemma~\ref{lem:upper-cutoff-error}]
Fix $x\ge\kappa\varepsilon^2\ge0$.  Let $Z_1,Z_2$ be
independent with CDF $F_\varepsilon$, and set
$\widetilde Z_i=Z_i\vee0$.  Then $\widetilde Z_1,\widetilde Z_2$ are
independent with CDF $G_\varepsilon$.  Since $F_\varepsilon$ is
continuous,
\[
 \Pp(Z_i<0)=F_\varepsilon(0)=\alpha_\varepsilon.
\]

For the parallel-composition map, the two input pairs agree when $Z_1,Z_2\ge0$.
If at least one original input is negative, then both
$g_-(Z_1,Z_2)$ and
$g_-(\widetilde Z_1,\widetilde Z_2)$ are negative.  Therefore the indicators
of the events $\{g_-(Z_1,Z_2)\le x\}$ and
$\{g_-(\widetilde Z_1,\widetilde Z_2)\le x\}$ coincide for $x\ge0$, and hence
\[
 P_{F_\varepsilon}(x)=P_{G_\varepsilon}(x).
\]

For the series-composition map, coordinatewise monotonicity gives
\[
 g_+(Z_1,Z_2)
 \le g_+(\widetilde Z_1,\widetilde Z_2).
\]
Thus
$S_{F_\varepsilon}(x)-S_{G_\varepsilon}(x)$ is the probability that
\[
 g_+(Z_1,Z_2)\le x
 <g_+(\widetilde Z_1,\widetilde Z_2).
\]
If $Z_1,Z_2<0$, then $g_+(Z_1,Z_2)<\log2$, whereas
$g_+(\widetilde Z_1,\widetilde Z_2)=\log2$.  This case therefore
contributes only when $0\le x<\log2$, and its contribution is at most
$\alpha_\varepsilon^2$.

Suppose next that $Z_1<0\le Z_2$.  Since
$g_+(Z_1,Z_2)\le x<g_+(0,Z_2),$
we have
$Z_2<g_+(Z_1,Z_2)\le x$, and hence $Z_2\le x$.  If
$0\le x<\log2$, this gives
$0\le Z_2\le x$.  If $x\ge\log2$, the additional inequality
\[
 x<g_+(0,Z_2)=\log(1+e^{Z_2})
\]
is equivalent to $\log(e^x-1)<Z_2$.  Define
\begin{equation*}
\begin{aligned}
 d(0)&=0,\\
 d(x)&=
 \begin{cases}
  x,&0<x<\log2,\\
  -\log(1-e^{-x}),&x\ge\log2,
 \end{cases}\\
 0&\le d(x)\le\log2.
\end{aligned}
\end{equation*}
Indeed, for $x\ge\log2$ we have
$\log(e^x-1)=x-d(x)$ and $1-e^{-x}\ge1/2$, which gives
$d(x)\le\log2$.  Independence and the same argument with the two
inputs exchanged now yield
\begin{equation*}
\begin{aligned}
 0\le S_{F_\varepsilon}(x)-S_{G_\varepsilon}(x)
 \le \alpha_\varepsilon^2\ind_{\{0\le x<\log2\}}+2\alpha_\varepsilon
 \bigl(F_\varepsilon(x)-F_\varepsilon(x-d(x))\bigr).
\end{aligned}
\end{equation*}

Put $\xi=z_\varepsilon+\varepsilon x$.  Since $q>0$,
\eqref{eq:full-line-relative-bounds} gives
$|(\log q)'|\le M_\lambda$.  Consequently,
\[
\begin{aligned}
 F_\varepsilon(x)-F_\varepsilon(x-d(x))
 &=\varepsilon\int_0^{d(x)}q(\xi-\varepsilon s)\,\dd s\\
 &\le\varepsilon d(x)e^{M_\lambda\varepsilon d(x)}q(\xi)
 \le C_\lambda\varepsilon q(\xi).
\end{aligned}
\]

Moreover, $\Phi(z_\varepsilon)=\varepsilon^4\to0$, so that
$z_\varepsilon+\varepsilon\log2\to-\infty$.  The left-tail ratio in
\eqref{eq:full-line-tail-ratios} therefore gives, for all sufficiently
small $\varepsilon$,
\[
 q(z_\varepsilon+\varepsilon x)
 \ge c_\lambda\Phi(z_\varepsilon+\varepsilon x)
 \ge c_\lambda\Phi(z_\varepsilon)
 =c_\lambda\alpha_\varepsilon,
 \qquad 0\le x\le\log2.
\]
Thus, for $0\le x\le\log2$, the inequality
$\alpha_\varepsilon^2\le
c_\lambda^{-1}\alpha_\varepsilon q(\xi)$ holds.  Using also
$p_+\le1$, $\varepsilon\le1$, and the equality of the parallel-composition
CDFs, we obtain
\[
\begin{aligned}
 0
 &\le \mathcal T_{p_+}F_\varepsilon(x)
       -\mathcal T_{p_+}G_\varepsilon(x)
 =p_+\bigl(S_{F_\varepsilon}(x)-S_{G_\varepsilon}(x)\bigr)\\
 &\le \alpha_\varepsilon^2\ind_{\{0\le x<\log2\}}
       +C_\lambda\alpha_\varepsilon\varepsilon q(\xi)
 \le C_\lambda\alpha_\varepsilon q(\xi)
 =C_\lambda\varepsilon^4q(z_\varepsilon+\varepsilon x).
\end{aligned}
\]
This proves \eqref{eq:upper-operator-cutoff-error}.
\end{proof}
\subsection{Lower bound}

Fix $0<\lambda<\lambda_*$ and $0<\kappa<\kappa_{\lambda}$.  It is enough to construct an
integrable \(Z_\varepsilon\) with CDF \(H_\varepsilon\) such that
\begin{align}
 &F_0\ge H_\varepsilon, \label{low_barrier_init}\\
 &\mathcal T_{p_-}H_\varepsilon(x)
 \ge H_\varepsilon(x+\kappa\varepsilon^2),
 \qquad x\in\mathbb R.\label{low_one_step}
\end{align}
Indeed, if such a CDF $H_\varepsilon$ exists, the same induction gives
\[
 F_n^-(x)\ge H_\varepsilon(x+n\kappa\varepsilon^2),
 \qquad
 X_n(p_-)\stle Z_\varepsilon-n\kappa\varepsilon^2.
\]
By integrability and Theorem~\ref{thm:logarithmic-speeds},
$v(p_-)\le-\kappa\varepsilon^2$.
By Lemma~\ref{lem:resistance-duality},
$v(p_+)=-v(p_-)$, and hence
\[
 v(p_+)\ge\kappa\varepsilon^2.
\]
Since $\delta=\varepsilon^3$, this gives
\[
 \liminf_{\delta\downarrow0}
 \frac{v(\frac12+\delta)}{\delta^{2/3}}
 \ge\kappa.
\]
Letting first $\kappa\uparrow\kappa_{\lambda}$ and then
$\lambda\uparrow\lambda_*$, and using \eqref{eq:beta-kappa},
gives the lower bound
\begin{equation}
 \liminf_{\delta\downarrow0}
 \frac{v(\frac12+\delta)}{\delta^{2/3}}
 \ge2\zeta(3)^{1/3}\lambda_*.
 \label{eq:sharp-lower-bound}
\end{equation}

It remains to construct an integrable $Z_\varepsilon$ whose CDF
$H_\varepsilon$ has the two properties in \eqref{low_barrier_init} and \eqref{low_one_step}.

\paragraph*{Step 1: Construct $H_\varepsilon$ and verify \eqref{low_barrier_init}.}
Let $\Psi=\Psi_{\lambda}$ and $q=\Psi'$ be as in
Proposition~\ref{prop:hard-edge-distribution}, using the zero extension of
$\Psi$, and define
\[
H_\varepsilon(x)=\Psi(\varepsilon x),\qquad x\in\mathbb R.
\]
This is an absolutely continuous CDF
with density
\[
\rho_\varepsilon(x)
=\varepsilon q(\varepsilon x)\mathbf 1_{\{x>0\}}.
\]
The density vanishes on $(-\infty,0)$
but has the positive right limit
$\rho_\varepsilon(0+)=\varepsilon q(0)>0$.
Since $H_\varepsilon$ is supported on $[0,\infty)$, we have $F_0\ge H_\varepsilon$.
Moreover, the finite-mean conclusion of
Proposition~\ref{prop:hard-edge-distribution} shows that the associated random
variable $Z_\varepsilon$ is integrable.  Thus only \eqref{low_one_step} remains to be verified.

\paragraph*{Step 2: Prove \eqref{low_one_step}.}
We start with a required  estimate for \eqref{low_one_step}.

\begin{lemma}
\label{lem:hard-edge-consistency}
There are $C_{\lambda}<\infty$ and $\varepsilon_{\lambda}>0$ such that, for
$0<\varepsilon<\varepsilon_{\lambda}$ and $x\ge0$,
\begin{equation}
 I_+(H_\varepsilon;x)-I_-(H_\varepsilon;x)
 \ge a\varepsilon^3q(\varepsilon x)q'(\varepsilon x)
      -C_{\lambda}\varepsilon^5q(\varepsilon x).
 \label{eq:hard-edge-one-sided-consistency}
\end{equation}
\end{lemma}

For $x\ge0$, Lemma~\ref{lem:hard-edge-consistency},
\eqref{eq:lower-profile-equation}, and the nonnegative last term in
\eqref{eq:exact-cdf-operator} give a one-step gain
\[
 \mathcal T_{p_-}H_\varepsilon(x)-H_\varepsilon(x)
 \ge \kappa_{\lambda}\varepsilon^3q(\varepsilon x)
      -C_{\lambda}\varepsilon^5q(\varepsilon x),
\]
whereas Taylor's theorem gives
\[
 H_\varepsilon(x+\kappa\varepsilon^2)-H_\varepsilon(x)
 =\kappa\varepsilon^3q(\varepsilon x)
  +O_{\lambda,\kappa}(\varepsilon^6q(\varepsilon x)).
\]
The strict inequality $\kappa<\kappa_{\lambda}$ therefore gives the desired
barrier on $[0,\infty)$. The remaining case $x<0$ is handled separately in
Proposition~\ref{prop:hard-edge-barrier}.

\begin{proposition}
\label{prop:hard-edge-barrier}
Fix $0<\lambda<\lambda_*$ and $0<\kappa<\kappa_{\lambda}$.  For all sufficiently small
$\varepsilon$,
\begin{equation*}
 \mathcal T_{p_-}H_\varepsilon(x)
 \ge H_\varepsilon(x+\kappa\varepsilon^2),
 \qquad x\in\mathbb R.
\end{equation*}
\end{proposition}

Proposition~\ref{prop:hard-edge-barrier} proves the remaining case, so $H_\varepsilon$ has the two properties used above.

We now prove Lemma~\ref{lem:hard-edge-consistency} and Proposition~\ref{prop:hard-edge-barrier}.

\begin{proof}[Proof of Lemma~\ref{lem:hard-edge-consistency}]
Fix $x\ge0$ and put $\xi=\varepsilon x$.  Let $\overline q$ be the
strictly positive $C^3$ extension from
\eqref{eq:q-extension-bounds}, and define
\begin{align*}
 \overline I_+(x)
 &=\varepsilon^2\int_0^\infty\int_0^{h(r)}
   \overline q\bigl(\xi+\varepsilon s\bigr)
   \overline q\bigl(\xi+\varepsilon(s+r)\bigr)
   \,\dd s\,\dd r,\\
 \overline I_-(x)
 &=\varepsilon^2\int_0^\infty\int_0^{h(r)}
   \overline q\bigl(\xi-\varepsilon s\bigr)
   \overline q\bigl(\xi-\varepsilon(s+r)\bigr)
   \,\dd s\,\dd r.
\end{align*}
The density $\rho_\varepsilon$ of $H_\varepsilon$ satisfies
\[
 0\le\rho_\varepsilon(y)
 \le\varepsilon\overline q(\varepsilon y),
 \qquad y\in\mathbb R,
\]
with equality for $y>0$.  Since $x+s$ and $x+s+r$ are nonnegative,
the integrands defining $I_+(H_\varepsilon;x)$ and
$\overline I_+(x)$ therefore agree almost everywhere, while the
integrand defining $I_-(H_\varepsilon;x)$ is bounded above by the
one defining $\overline I_-(x)$.  Consequently,
\begin{equation}
 I_+(H_\varepsilon;x)=\overline I_+(x),
 \qquad
 I_-(H_\varepsilon;x)\le\overline I_-(x),
 \label{eq:hard-edge-extension-comparison}
\end{equation}
and hence
\[
 I_+(H_\varepsilon;x)-I_-(H_\varepsilon;x)
 \ge \overline I_+(x)-\overline I_-(x).
\]

For fixed $r\ge0$ and $0\le s\le h(r)$, set
\[
 \overline Q_{s,r}(t)
 =\overline q(\xi+ts)\,
  \overline q(\xi+t(s+r)).
\]
Then
\begin{equation}
 \overline I_+(x)-\overline I_-(x)
 =\varepsilon^2\int_0^\infty\int_0^{h(r)}
 \bigl(\overline Q_{s,r}(\varepsilon)
       -\overline Q_{s,r}(-\varepsilon)\bigr)
 \,\dd s\,\dd r.
 \label{eq:hard-edge-extension-difference}
\end{equation}
Applying Taylor's theorem at $0$ to $t=\pm\varepsilon$ and
subtracting gives
\begin{equation}
\begin{aligned}
 \overline Q_{s,r}(\varepsilon)
 -\overline Q_{s,r}(-\varepsilon)
 &=2\varepsilon\overline Q_{s,r}'(0)+\overline R_{s,r}
 =2\varepsilon(2s+r)
   \overline q(\xi)\overline q'(\xi)+\overline R_{s,r},\\
 |\overline R_{s,r}|
 &\le\frac{\varepsilon^3}{3}
 \sup_{|t|\le\varepsilon}|\overline Q_{s,r}'''(t)|.
\end{aligned}
 \label{eq:hard-edge-symmetric-Taylor}
\end{equation}

We next estimate the remainder.  Write
$\overline M=\overline M_\lambda$, and allow $C_\lambda$ to change
from line to line.  Since $\overline q>0$, the case $j=1$ in
\eqref{eq:q-extension-bounds} gives
$|(\log\overline q)'|\le\overline M$, and hence
\begin{equation}
 e^{-\overline M|y|}\overline q(z)
 \le \overline q(z+y)
 \le e^{\overline M|y|}\overline q(z),
 \qquad z,y\in\mathbb R.
 \label{eq:extension-log-Lipschitz}
\end{equation}
Moreover, since $\overline q=q$ on $[0,\infty)$ and $q$ is the
density of $\Psi$, we have
\[
 \int_0^\infty\overline q(u)\,\dd u=1.
\]
Thus, for every $\xi\ge0$, \eqref{eq:extension-log-Lipschitz} gives
\[
 1\ge\int_\xi^{\xi+1}\overline q(u)\,\dd u
 \ge e^{-\overline M}\overline q(\xi),
\]
and therefore
\begin{equation}
 \overline q(\xi)^2
 \le e^{\overline M}\overline q(\xi),
 \qquad \xi\ge0.
 \label{eq:extension-positive-halfline-bound}
\end{equation}
Choose $\varepsilon_\lambda>0$ so that
$2\overline M\varepsilon_\lambda\le1/2$, and assume henceforth that
$0<\varepsilon<\varepsilon_\lambda$.  Since
$0\le s\le h(r)\le\log2<1$, we have $2s+r\le2(1+r)$.  Hence
Leibniz's rule,
\eqref{eq:q-extension-bounds}, and
\eqref{eq:extension-log-Lipschitz} imply, for all
$|t|\le\varepsilon$,
\[
 \begin{aligned}
 |\overline Q_{s,r}'''(t)|
 &\le C_\lambda
 \sum_{j=0}^3\binom3j
 s^j(s+r)^{3-j}
 \overline q(\xi+ts)
 \overline q(\xi+t(s+r))\\
 &=C_\lambda(2s+r)^3
 \overline q(\xi+ts)
 \overline q(\xi+t(s+r))\\
 &\le C_\lambda(1+r)^3
 e^{2\overline M\varepsilon(1+r)}
 \overline q(\xi)^2.
 \end{aligned}
\]
The contribution of $\overline R_{s,r}$ to
\eqref{eq:hard-edge-extension-difference} satisfies
\[
 \begin{aligned}
 \left|
 \varepsilon^2\int_0^\infty\int_0^{h(r)}
 \overline R_{s,r}\,\dd s\,\dd r
 \right|
 &\le C_\lambda\varepsilon^5\overline q(\xi)^2
 \int_0^\infty h(r)(1+r)^3
 e^{2\overline M\varepsilon(1+r)}\,\dd r\\
 &\le C_\lambda\varepsilon^5\overline q(\xi)^2
 \int_0^\infty(1+r)^3e^{-r/2}\,\dd r\\
 &\le C_\lambda\varepsilon^5\overline q(\xi).
 \end{aligned}
\]
Here the second inequality uses $h(r)\le e^{-r}$, with the factor
$e^{1/2}$ absorbed into $C_\lambda$, and the last one uses
\eqref{eq:extension-positive-halfline-bound}.

The linear term in \eqref{eq:hard-edge-symmetric-Taylor} contributes
\[
 2\varepsilon^3\overline q(\xi)\overline q'(\xi)
 \int_0^\infty\int_0^{h(r)}(2s+r)\,\dd s\,\dd r
 =a\varepsilon^3\overline q(\xi)\overline q'(\xi)
\]
by Lemma~\ref{lem:diffusion-coefficient}.  Consequently,
\begin{equation}
 \left|
 \overline I_+(x)-\overline I_-(x)
 -a\varepsilon^3\overline q(\xi)\overline q'(\xi)
 \right|
 \le C_\lambda\varepsilon^5\overline q(\xi).
 \label{eq:hard-edge-extended-consistency}
\end{equation}

Finally, since $\xi\ge0$ and $\overline q$ agrees with $q$ on
$[0,\infty)$,
\[
 \overline q(\xi)=q(\xi),
 \qquad
 \overline q'(\xi)=q'(\xi),
\]
where $q'(0)$ denotes the right derivative $q'(0+)$.  Combining
\eqref{eq:hard-edge-extension-comparison} and
\eqref{eq:hard-edge-extended-consistency}, and recalling that
$\xi=\varepsilon x$, yields
\[
 I_+(H_\varepsilon;x)-I_-(H_\varepsilon;x)
 \ge a\varepsilon^3q(\varepsilon x)q'(\varepsilon x)
     -C_\lambda\varepsilon^5q(\varepsilon x),
\]
which proves \eqref{eq:hard-edge-one-sided-consistency}.
\end{proof}

\begin{proof}[Proof of Proposition~\ref{prop:hard-edge-barrier}]
Put $s_\varepsilon=\kappa\varepsilon^2$.
If $x<-s_\varepsilon$, the right-hand side is zero. The desired inequality is therefore immediate because
$\mathcal T_{p_-}H_\varepsilon$ is a CDF.

Suppose $x\ge0$ and write $\xi=\varepsilon x$.  Since the last term in
\eqref{eq:exact-cdf-operator} is nonnegative when $p=p_-$,
Lemma~\ref{lem:hard-edge-consistency} and
\eqref{eq:lower-profile-equation} give
\begin{align*}
 \mathcal T_{p_-}H_\varepsilon(x)-H_\varepsilon(x)
 &\ge \varepsilon^3
 \bigl(aq(\xi)q'(\xi)+2\Psi(\xi)(1-\Psi(\xi))\bigr)
 -C_{\lambda}\varepsilon^5q(\xi)\\
 &=\kappa_{\lambda}\varepsilon^3q(\xi)
 -C_{\lambda}\varepsilon^5q(\xi).
\end{align*}
On the other hand, by \eqref{eq:subcritical-relative-bounds} and Taylor's
theorem,
\[
 H_\varepsilon(x+s_\varepsilon)-H_\varepsilon(x)
 =\kappa\varepsilon^3q(\xi)+O_{\lambda,\kappa}(\varepsilon^6q(\xi)).
\]
Because $\kappa<\kappa_{\lambda}$, the desired inequality follows for all
$x\ge0$ when $\varepsilon$ is small.

It remains to consider $-s_\varepsilon\le x<0$.  Here
$H_\varepsilon(x)=0$ and $I_-(H_\varepsilon;x)=0$, so
\[
 \mathcal T_{p_-}H_\varepsilon(x)
 =(1+2\varepsilon^3)I_+(H_\varepsilon;x)
 \ge I_+(H_\varepsilon;x).
\]
Let $q_0=q(0+)>0$.  Since $q(z)\to q_0$ as $z\downarrow0$, choose
$\eta>0$ such that
 $q(z)\ge\frac{q_0}{2}$ for $ 0\le z\le\eta$.
Set $b=h(1)/4$, and take $\varepsilon$ sufficiently small that
\[
 s_\varepsilon<b/2,
 \qquad
 \varepsilon(3b+1)\le\eta.
\]
Since $h$ is decreasing, for $0\le r\le1$ we have
$h(r)\ge h(1)=4b$, so the rectangle
\[
 0\le r\le1,
 \qquad
 2b\le s\le3b
\]
is contained in the integration region
$\{r\ge0,\ 0\le s\le h(r)\}$.  Moreover, for
$-s_\varepsilon\le x<0$ and $(r,s)$ in this rectangle,
\[
 0<2b-s_\varepsilon
 \le x+s\le x+s+r\le3b+1.
\]
It follows that
\[
 \rho_\varepsilon(x+s)
 \ge\frac{\varepsilon q_0}{2},
 \qquad
 \rho_\varepsilon(x+s+r)
 \ge\frac{\varepsilon q_0}{2}.
\]
Restricting $I_+(H_\varepsilon;x)$ to this rectangle therefore gives
\begin{equation}
 \begin{aligned}
 I_+(H_\varepsilon;x)
 &\ge\int_0^1\int_{2b}^{3b}
   \rho_\varepsilon(x+s)
   \rho_\varepsilon(x+s+r)\,\dd s\,\dd r\\
 &\ge\int_0^1\int_{2b}^{3b}
   \frac{\varepsilon^2q_0^2}{4}\,\dd s\,\dd r
 =\frac{bq_0^2}{4}\varepsilon^2.
 \end{aligned}
 \label{eq:boundary-layer-positive-mass}
\end{equation}

Since $q(z)\to q_0$ as $z\downarrow0$, we have
$q(z)\le2q_0$ for all sufficiently small $z>0$.  Moreover, on the
present strip, $0\le x+s_\varepsilon\le s_\varepsilon$.  Hence, for all
sufficiently small $\varepsilon$,
\[
\begin{aligned}
 H_\varepsilon(x+s_\varepsilon)
 =\Psi\bigl(\varepsilon(x+s_\varepsilon)\bigr)
 =\int_0^{\varepsilon(x+s_\varepsilon)}q(z)\,\dd z
 \le2q_0\varepsilon(x+s_\varepsilon)
 \le2q_0\varepsilon s_\varepsilon
 =2\kappa q_0\varepsilon^3.
\end{aligned}
\]
  The lower bound
\eqref{eq:boundary-layer-positive-mass} dominates this quantity, completing
the proof.
\end{proof}

\begin{proof}[Proof of Theorem~\ref{thm:near-critical-speed}]
The existence, positivity, and minimality of $\lambda_*$ for the
boundary-value problem \eqref{eq:main-W-bvp} follow from
Proposition~\ref{prop:Cstar-halfline}.
  Combining \eqref{eq:sharp-upper-bound} and
\eqref{eq:sharp-lower-bound}, and using
Lemma~\ref{lem:resistance-duality}, yields
\eqref{eq:main-near-critical-limit}.
For the first-moment assertions, Theorem~\ref{thm:first-moment-logarithmic-rates}
gives
$\gamma_R(1/2+\delta)=v_R(1/2+\delta)$.  On the subcritical side, since
$\delta^{2/3}\gg\delta$, \eqref{eq:main-near-critical-limit},
Theorem~\ref{thm:first-moment-logarithmic-rates}, and
$\log(1-2\delta)\sim-2\delta$ give, for all sufficiently small $\delta>0$,
\[
 \gamma_R\left(\frac12-\delta\right)
 =v_R\left(\frac12-\delta\right)\vee\log(1-2\delta)
 =\log(1-2\delta).
\]
Together with \eqref{eq:main-near-critical-limit}, these identities yield
\eqref{eq:resistance-first-moment-near-critical}.  This completes the proof.
\end{proof}

\begin{remark}[Why the exponent $2/3$?]
\label{rem:two-thirds-heuristic}
The exponent $2/3$ is suggested by a balance of scales. Let
$K_\eta(x)=K(\eta x)$, where $K$ is a smooth CDF for which the one-step
expansion used above is valid. The first nonzero symmetric contribution to the one-step
update has order $\eta^3$, whereas the bias contributes at order $\delta$.
Balancing these terms gives $\eta^3$ of order $\delta$. A translation by $s$
changes $K_\eta$ at order $s\eta$, so matching the translation to the one-step
update gives $s\eta$ of order $\delta$. Thus $\eta$ has order
$\delta^{1/3}$ and $s$ has order $\delta^{2/3}$. The upper and lower comparison
arguments above make this scaling prediction rigorous for the logarithmic speed.
\end{remark}

\appendix
\section{ODE analysis and admissible parameters}
\label{app:ode}
In this appendix, we study the ordinary differential equation (ODE) for \(W\in C([0,1])\cap C^1((0,1))\) with parameter \(\lambda > 0\),
\begin{equation}
 W^2W'-\lambda W+u(1-u)=0,
 \qquad 0<u<1,
 \label{eq:W-ode}
\end{equation}
subject to
\begin{equation}
 W(0)=W(1)=0,
 \qquad W>0\quad\text{on }(0,1).
 \label{eq:W-boundary}
\end{equation}

This appendix provides a self-contained proof of the ODE classification needed in Section~\ref{sec:two-sided-barriers} that is a special case of known $p$-Laplacian travelling-wave theory. Setting $Y(u)=W(u)^3$ transforms the equation into
\[
Y'=\frac32(2\lambda Y^{1/3}-2u(1-u)),
\]
which is the case $p=3$, $q=3/2$, $c=2\lambda$, and $f(u)=2u(1-u)$
of the equation studied by Engui\c{c}a, Gavioli, and
Sanchez~\cite[Proposition~2 and Theorem~3.3]{EnguicaGavioliSanchez2013}. Equivalently, it is the phase-plane equation for the $3$-Laplacian Fisher--KPP travelling waves studied by Audrito and Vázquez~\cite[Theorem~2.1 and Section~3]{AudritoVazquez2017}.

In the present normalization, we show that there exists a least admissible parameter $\lambda_*>0$: the
boundary-value problem has a unique positive solution for every
$\lambda\ge\lambda_*$ and no such solution for $0<\lambda<\lambda_*$.  We
then determine the possible endpoint asymptotics of these solutions at
$u=0$. The existence and uniqueness classification is given by
\cite[Proposition~2]{EnguicaGavioliSanchez2013}, and the critical and
supercritical branches are identified in
\cite[Theorem~3.3]{EnguicaGavioliSanchez2013}. Our analysis determines the
precise linear asymptotic in the supercritical branch and the differentiable
remainder estimates needed in Appendix~\ref{app:regularity}.

\subsection{Admissible parameters of the ODE}

For $t\in[0,1]$, let $y(t)=(W(1-t))^3$.  Under this substitution,
\eqref{eq:W-ode} is equivalent to
\begin{equation}
 y'=3t(1-t)-3\lambda y^{1/3}, \qquad 0<t<1.
 \label{eq:y-ODE}
\end{equation}
The boundary conditions \eqref{eq:W-boundary} are equivalent to
\begin{equation}
    \label{eq:y-boundary}
 y(0)=y(1)=0, \qquad y>0 \quad\text{on }(0,1).
\end{equation}
We first study \eqref{eq:y-ODE} with the initial condition $y(0)=0$:
\begin{equation}
 y'=3t(1-t)-3\lambda y^{1/3}, \qquad 0<t<1,\qquad y(0)=0.
 \label{eq:y-shooting}
\end{equation}
We prove existence and uniqueness for \eqref{eq:y-shooting}, and then
determine the parameters $\lambda>0$ for which \eqref{eq:y-boundary} also
holds.

Cube roots throughout this appendix are understood as real cube roots.
Local existence for \eqref{eq:y-shooting} follows from Peano's theorem, since
the right-hand side is continuous.  To obtain existence on all of \([0,1]\) and
uniqueness, we use the following comparison principle; this is needed because
the cube-root term is not locally Lipschitz at zero.

\begin{lemma}
\label{lem:cuberoot-comparison}
Let \(I=[t_0,t_1]\), \(a\in C(I)\) and \(b>0\). Suppose that
\(\underline y,\overline y\in C(I)\cap C^1(\mathring{I})\) satisfy
\[
 \underline y'\le a(t)-b\underline y^{1/3},
 \qquad
 \overline y'\ge a(t)-b\overline y^{1/3}.
\]
If
\(\underline y(t_0)\le\overline y(t_0)\), then
\(\underline y\le\overline y\) on \(I\).
\end{lemma}

\begin{proof}
Let \(d:=\underline y-\overline y\) and
\(\Psi(s)=\frac12(s_+)^2\).  For every \(t\in\mathring I\), the chain
rule gives
\[
 \frac{\dd}{\dd t}\Psi(d(t))
 \le -b d_+(t)
 \bigl(\underline y(t)^{1/3}-\overline y(t)^{1/3}\bigr)
 \le 0.
\]
Thus \(\Psi(d)\) is non-increasing on every compact subinterval of
$\mathring I$.  By continuity at the endpoints and
\(\Psi(d(t_0))=0\), it follows that \(\Psi(d)=0\) throughout \(I\), and hence
\(\underline y\le\overline y\) on \(I\).
\end{proof}

\begin{lemma}
\label{lem:shooting-properties}
For every \(\lambda>0\), equation \eqref{eq:y-shooting} has a unique solution
\(y_\lambda\) on \([0,1]\), and
\begin{equation}
    \label{eq:y>0}
y_\lambda(t)>0,\qquad 0<t<1.
\end{equation}
  If
\(\lambda_2>\lambda_1\), then
\begin{equation}
 y_{\lambda_2}(t)<y_{\lambda_1}(t),
 \qquad 0<t<1.
 \label{eq:shooting-strict-order}
\end{equation}
Finally, \(\lambda\mapsto y_\lambda\) is continuous from \((0,\infty)\) into \(C([0,1])\).

\end{lemma}

\begin{proof}
For fixed
$\lambda>0$, the function
\[
  f_\lambda(t,y)=3t(1-t)-3\lambda\sqrt[3]{y}
\]
is continuous on $[0,1]\times\mathbb R$ and satisfies
\[
  |f_\lambda(t,y)|
  \le \frac34+3\lambda+3\lambda|y|.
\]
The global form of Peano's existence theorem on a compact time interval,
under a linear-growth bound, therefore gives a solution of \eqref{eq:y-shooting} on
$[0,1]$. Since the zero function is a subsolution, Lemma~\ref{lem:cuberoot-comparison} yields
$y_\lambda\ge0$. Consequently,
\[
  y_\lambda'(t)\le3t(1-t),\qquad
  0\le y_\lambda(t)\le\frac32t^2-t^3\le\frac12
  \quad(0\le t\le1).
\]
If $y_\lambda(t_0)=0$ for some $t_0\in(0,1)$, then $t_0$ is a local
minimum, so $y_\lambda'(t_0)=0$. This contradicts
$y_\lambda'(t_0)=3t_0(1-t_0)>0$. Hence $y_\lambda>0$ on $(0,1)$.
Applying Lemma~\ref{lem:cuberoot-comparison} in both directions to two solutions with the same
initial condition proves uniqueness.

Let $\lambda_2>\lambda_1>0$. The solution $y_{\lambda_2}$ is a
subsolution of the equation with parameter $\lambda_1$, because it is
nonnegative. Thus Lemma~\ref{lem:cuberoot-comparison} gives
$y_{\lambda_2}\le y_{\lambda_1}$ on $[0,1]$. Set
$d=y_{\lambda_1}-y_{\lambda_2}$. If $d(t_0)=0$ for some
$t_0\in(0,1)$, then $d'(t_0)=0$, whereas the equations imply
\[
  d'(t_0)
  =3(\lambda_2-\lambda_1)y_{\lambda_2}(t_0)^{1/3}>0.
\]
This contradiction proves the strict inequality in \eqref{eq:shooting-strict-order}.

Finally, using the preceding order relation and the bound
$y_{\lambda_2}\le1/2$, we obtain, for $0<t<1$,
\begin{align*}
  d'(t)
  &=-3\lambda_1\bigl(y_{\lambda_1}(t)^{1/3}
                    -y_{\lambda_2}(t)^{1/3}\bigr)
    +3(\lambda_2-\lambda_1)y_{\lambda_2}(t)^{1/3} \\
  &\le 3\,2^{-1/3}(\lambda_2-\lambda_1).
\end{align*}
Since $d(0)=0$, integration and continuity at the endpoints give
\[
  0\le d(t)\le 3\,2^{-1/3}(\lambda_2-\lambda_1)t
  \quad(0\le t\le1).
\]
Interchanging the parameters when necessary shows that
\[
  \|y_{\lambda_1}-y_{\lambda_2}\|_{C([0,1])}
  \le 3\,2^{-1/3}|\lambda_1-\lambda_2|
  \quad(\lambda_1,\lambda_2>0).
\]
In particular, $\lambda\mapsto y_\lambda$ is continuous from
$(0,\infty)$ into $C([0,1])$.
\end{proof}

For later comparison constructions, it is convenient to introduce the
following function.  For a differentiable nonnegative function $w$ on an interval, set
\[
 \mathcal R_\lambda[w](t)=w(t)^2w'(t)+\lambda w(t)-t(1-t).
\]
If \(z=w^3\), then
\[
 z'-\bigl(3t(1-t)-3\lambda z^{1/3}\bigr)=3\mathcal R_\lambda[w].
\]
Thus \(\mathcal R_\lambda[w]\le0\) corresponds to a subsolution of
\eqref{eq:y-shooting}, and \(\mathcal R_\lambda[w]\ge0\) to a supersolution.

Define the admissible set
\[
 \mathcal A:=\{\lambda>0:y_\lambda(1)=0\}.
\]
By Lemma~\ref{lem:shooting-properties}, \(\lambda\in\mathcal A\) if and only if
$W_\lambda(u)=\bigl[y_\lambda(1-u)\bigr]^{1/3}$
solves \eqref{eq:W-ode} and \eqref{eq:W-boundary}; for a fixed \(\lambda\), this positive
solution is unique.

We now prove Proposition~\ref{prop:Cstar-halfline} and obtain the stated explicit lower bound.

\begin{proof}[Proof of Proposition~\ref{prop:Cstar-halfline}]
We first show that \(\mathcal A\) is nonempty. Take \(\overline w(t):=\frac{3}{2\lambda}t(1-t)\).
Then
\[
 \mathcal R_\lambda[\overline w]
 =t(1-t)\left(\frac{1}{2}+\left(\frac{3}{2\lambda}\right)^3t(1-t)(1-2t)\right).
\]
Notice
\[
 \max_{0\le t\le1}\{-t(1-t)(1-2t)\}=\frac1{6\sqrt3},
\]
so if \(\lambda\ge\sqrt3/2\), we have
\[
\frac{1}{2}+\left(\frac{3}{2\lambda}\right)^3t(1-t)(1-2t) \ge 0.
\]
Hence
\(\mathcal R_\lambda[\overline w]\ge0\) on \([0,1]\).  Therefore \(\overline y:=\overline w^3\) is a supersolution of
\eqref{eq:y-shooting}.  Since $\overline y(1)=0$, the comparison
principle (Lemma~\ref{lem:cuberoot-comparison}) yields
\(y_\lambda(1)=0\). This is equivalent to \(\lambda\in\mathcal A\).

By Lemma~\ref{lem:shooting-properties}, $y_\lambda$ is monotone in
$\lambda$.  If $\lambda_1\in\mathcal A$ and $\lambda_2>\lambda_1$, then
\[
 0\le y_{\lambda_2}(1)\le y_{\lambda_1}(1)=0,
\]
so $\lambda_2\in\mathcal A$.  Thus $\mathcal A$ is upward closed.  Let
$\lambda_*:=\inf\mathcal A$.

We next prove that $\lambda_*>0$.  Fix $\lambda\in\mathcal A$.  Dividing
\eqref{eq:W-ode} by $W_\lambda$ for $u\in(0,1)$ gives
\begin{equation}
 (W_\lambda^2(u))'
 =2\lambda-2\frac{u(1-u)}{W_\lambda(u)}
 \le2\lambda.
 \label{eq:ODE divide W}
\end{equation}
Since $W_\lambda(0)=0$, direct integration gives
$W_\lambda(u)^2\le2\lambda u$, that is,
$W_\lambda(u)\le\sqrt{2\lambda u}$ for $0\le u\le1$.

To justify the endpoint integral, integrate the identity in
\eqref{eq:ODE divide W} over $[\rho,1-\rho]$, where $0<\rho<1/2$:
\[
 W_\lambda(1-\rho)^2-W_\lambda(\rho)^2
 =2\lambda(1-2\rho)
  -2\int_\rho^{1-\rho}
      \frac{u(1-u)}{W_\lambda(u)}\,\dd u.
\]
As $\rho\downarrow0$, the left-hand side tends to zero by continuity and
the boundary conditions.  The integral converges by monotone convergence,
and therefore
\begin{equation}
 \lambda=\int_0^1\frac{u(1-u)}{W_\lambda(u)}\,\dd u.
 \label{eq:lambda-integral-W}
\end{equation}
Substituting $W_\lambda(u)\le\sqrt{2\lambda u}$ into
\eqref{eq:lambda-integral-W} gives
\[
 \lambda\ge\frac1{\sqrt{2\lambda}}
 \int_0^1u^{1/2}(1-u)\,\dd u
 =\frac4{15\sqrt{2\lambda}}.
\]
This proves the lower bound for $\lambda_*$, and hence
$\lambda_*>0$.

Finally, Lemma~\ref{lem:shooting-properties} shows that
$\lambda\mapsto y_\lambda$ is continuous on $(0,\infty)$.  Since
$\lambda_*>0$, a sequence in $\mathcal A$ decreasing to $\lambda_*$ gives
$y_{\lambda_*}(1)=0$, so $\lambda_*\in\mathcal A$.  Combining this with $\mathcal{A}$ being upward closed proves $\mathcal A=[\lambda_*,\infty)$ and the upper bound on $\lambda_*$.
\end{proof}

\subsection{Asymptotics of \texorpdfstring{\(W_{\lambda}\)}{W(lambda)} at endpoints}

For $0<\lambda<\lambda_*$, Lemma~\ref{lem:shooting-properties} and
Proposition~\ref{prop:Cstar-halfline} imply that $y_\lambda(1)>0$.
Hence the solution $W_\lambda$ satisfying $W_\lambda(1)=0$ has
$W_\lambda(0)>0$ and does not satisfy both boundary conditions in
\eqref{eq:W-boundary}.  For $\lambda\ge\lambda_*$, the boundary-value
problem has a unique positive solution.  We now determine the endpoint
asymptotics in these two regimes.

\begin{proposition}
\label{prop:subcritical-W}
If \(0<\lambda<\lambda_*\), then
\begin{equation}
 y_\lambda(1)>0,
 \qquad
 W_\lambda(u):=\bigl[y_\lambda(1-u)\bigr]^{1/3}>0
 \qquad\text{for }0\le u<1,
 \label{eq:subcritical-positive-W}
\end{equation}
with \(W_\lambda(1)=0\) and \(W_\lambda(0)>0\).  The function \(W_\lambda\) solves
\eqref{eq:W-ode} on \((0,1)\) and is smooth at \(u=0\).
\end{proposition}

\begin{proof}
Equation \eqref{eq:subcritical-positive-W} is a direct corollary of Lemma~\ref{lem:shooting-properties} and Proposition~\ref{prop:Cstar-halfline}.

It remains to prove that $W_{\lambda}$ is smooth at $u=0$.  For
$u\in[0,1)$, we have $W_\lambda(u)>0$.  Dividing both sides of
\eqref{eq:W-ode} by $W_\lambda(u)^2$ gives
\[
 W_\lambda'(u)
 =\frac{\lambda W_\lambda(u)-u(1-u)}{W_\lambda(u)^2}
 =:F(u,W_\lambda(u)).
\]
Since $W_\lambda$ is continuous and $W_\lambda(0)>0$, there are $c>0$ and
$\eta>0$ such that $W_\lambda(u)\ge c$ for $0\le u\le\eta$.  The function
\[
 F(u,w)=\frac{\lambda w-u(1-u)}{w^2}
\]
is smooth on a neighborhood of
$\{(u,W_\lambda(u)):u\in[0,\eta]\}$.  In particular,
\[
 W_\lambda'(u)=F(u,W_\lambda(u))
 \longrightarrow F(0,W_\lambda(0))
 =\frac{\lambda}{W_\lambda(0)}
 \qquad\text{as }u\downarrow0.
\]
The derivative is therefore bounded near zero.  Integrating first over
$[\delta,u]$ and then letting $\delta\downarrow0$ yields
\[
 W_\lambda(u)-W_\lambda(0)
 =\int_0^u F(s,W_\lambda(s))\,\dd s.
\]
Thus $W_\lambda$ extends to a $C^1$ function at zero, with
$W_\lambda'(0)=\lambda/W_\lambda(0)$.  Finally, if $W_\lambda$ is $C^k$ on
$[0,\eta]$, then the smoothness of $F$ shows that
$W_\lambda'=F(\cdot\,,W_\lambda)$ is $C^k$, and hence $W_\lambda$ is
$C^{k+1}$.  Induction proves that $W_\lambda$ is smooth at $u=0$.
\end{proof}

\begin{lemma}
\label{lem:linear-u1}
For every \(\lambda>0\),
\[
 W_\lambda(u)\sim\frac{1-u}{\lambda},
 \qquad u\uparrow1.
\]
\end{lemma}

\begin{proof}
Set
\[
 w_\lambda(t):=W_\lambda(1-t)=y_\lambda(t)^{1/3},
 \qquad 0\le t<1.
\]
Thus it suffices to prove that \(w_\lambda(t)/t\to1/\lambda\) as
\(t\downarrow0\).  Fix constants \(0<k_-<\frac1\lambda<k_+\),
and define \(\underline w(t)=k_-t\), \(\overline w(t)=k_+t\).  By direct computation, for any \(k>0\),
\[
 \mathcal R_\lambda[kt]
 =t(\lambda k-1)+t^2(k^3+1).
\]
Since $\lambda k_--1<0$ and $\lambda k_+-1>0$, there exists a
sufficiently small $\delta>0$ such that
$\underline y:=\underline w^3$ and
$\overline y:=\overline w^3$ are a subsolution and a supersolution,
respectively, of \eqref{eq:y-shooting} on $[0,\delta]$.  All three functions
have the same initial value \(0\), so Lemma~\ref{lem:cuberoot-comparison} gives
\[
 (k_-t)^3\le y_\lambda(t)\le(k_+t)^3,
 \qquad 0\le t\le\delta.
\]
Then it follows that
\[
 k_-\le
 \liminf_{t\downarrow0}\frac{w_\lambda(t)}t
 \le
 \limsup_{t\downarrow0}\frac{w_\lambda(t)}t
 \le k_+.
\]
Letting
\(k_-\uparrow1/\lambda\) and \(k_+\downarrow1/\lambda\) yields
\(w_\lambda(t)/t\to1/\lambda\) as
$t\downarrow0$.  Setting $u=1-t$ gives the desired result.
\end{proof}

We next classify the asymptotics at \(u=0\) for \(\lambda\in\mathcal A\).

\begin{lemma}
\label{lem:u0-dichotomy}
If \(\lambda\in\mathcal A\), exactly one of the following asymptotics holds:
\begin{align}
 W_\lambda(u)&\sim\frac{u}{\lambda},
 &&u\downarrow0,
 \label{eq:u0-linear}\\
 W_\lambda(u)&\sim\sqrt{2\lambda u},
 &&u\downarrow0.
 \label{eq:u0-sharp}
\end{align}
\end{lemma}

\begin{proof}
Write \(W=W_\lambda\) and
\[
 H(u)=\lambda W(u)-u(1-u)=W(u)^2W'(u).
\]
Since \(W(0)=0<W(u)\), the mean-value theorem shows that \(H>0\) at least on a sequence of points
tending to zero.  At any zero \(v\in(0,1/2)\) of \(H\),
\(W'(v)=0\) and
\[
 H'(v)=-(1-2v)<0.
\]
Thus $H$ has at most one zero in $(0,1/2)$.  Indeed, after a zero,
$H$ is negative
locally, whereas at the first return to zero its derivative would have to be
nonnegative.  Hence \(H>0\) in a sufficiently small interval \((0, u_0]\), and \(W\) is strictly increasing on \([0, u_0]\).

Consequently, $W$ has a continuous inverse
$U:[0,x_0]\to[0,u_0]$, where $x_0=W(u_0)$.  Moreover,
\[
 U\in C([0,x_0])\cap C^1((0,x_0]),
 \qquad U(0)=0,
 \qquad U(x)>0\quad(0<x\le x_0).
\]
Let \(r(x):=U(x)/x\).  The original equation \eqref{eq:W-ode} gives, for \(x\in (0, x_0)\),
\begin{equation}
\begin{aligned}
 U'(x)
 &=\frac{x^2}{\lambda x-U(x)+U(x)^2}
 =\frac{x}{\lambda-r(x)+xr(x)^2},\\
 xr'(x)
 &=\frac{x}{\lambda-r(x)+xr(x)^2}-r(x).
\end{aligned}
 \label{eq:r-inverse}
\end{equation}
Since \(U'(x)>0\), the denominator \(\lambda-r(x)+xr(x)^2\) is positive.  If \(r(x_j)\ge2\lambda\) for some
\(x_j\downarrow0\), then
\[
 x_jr(x_j)^2>r(x_j)-\lambda\ge\frac12r(x_j),
 \qquad U(x_j)=x_jr(x_j)>\frac12,
\]
contrary to \(U(x)\to0\).  Thus \(r\) is bounded near zero.  Since
\(r<\lambda+xr^2\), we also have
\begin{equation}
 \limsup_{x\downarrow0}r(x)\le \lambda.
 \label{eq:r-limsup}
\end{equation}

Let \(\rho(s):=r(e^{-s})\), \(s\in[-\log x_0,\infty)\).  Equation \eqref{eq:r-inverse} becomes
\begin{equation}
 \rho'(s)=\rho(s)-
 \frac{e^{-s}}{\lambda-\rho(s)+e^{-s}\rho(s)^2}.
 \label{eq:r-flow}
\end{equation}
We show that either \(\lim_{s\to\infty}\rho(s)=0\) or
\(\lim_{s\to\infty}\rho(s)=\lambda\).

We use the following elementary observation.  Suppose that there exists a sufficiently large \(S\)
such that \(\phi'(s)>0\) whenever \(s\ge S\) and \(\phi(s)=a\).  Then, for every
\(s_0\ge S\) with \(\phi(s_0)\ge a\), one has
\[
 \phi(s)\ge a\qquad\text{for all }s\ge s_0.
\]
Indeed, otherwise at the first time when \(\phi\) falls below \(a\), its
derivative at the boundary level \(a\) would be non-positive.

Suppose first that \(\liminf\rho=0\).  If \(\limsup\rho>0\), choose
\(0<\eta<\min\{\lambda/2,\limsup\rho\}\).  At \(\rho(s)=\eta\), for all large \(s\),
\[
 \rho'(s)\ge\eta-\frac{2e^{-s}}\lambda\ge\frac\eta2>0.
\]
Since \(\limsup\rho>\eta\), the observation gives \(\rho\ge\eta\) eventually,
contrary to \(\liminf\rho=0\).  Hence \(\rho(s)\to0\).

Now suppose \(L:=\liminf\rho>0\).  Fix \(\delta\in(0,\lambda)\) and choose
\(\eta\in(0,L/2)\).  There exists \(S\) such that, for every \(s\ge S\),
\[
 \rho(s)\ge\eta,
 \qquad
 \frac{e^{-s}}\delta\le\frac\eta2.
\]
Hence, whenever \(s\ge S\) and \(\rho(s)\le\lambda-\delta\),
\eqref{eq:r-flow} gives
\[
 \rho'(s)
 \ge\eta-\frac{e^{-s}}\delta
 \ge\frac\eta2.
\]
This lower bound is uniform for $s\ge S$ and
$\rho(s)\le\lambda-\delta$.  Hence $\rho$ reaches
$\lambda-\delta$ in finite time.  The preceding observation then gives
\(\rho(s)\ge\lambda-\delta\) for all sufficiently large \(s\).  Therefore
\[
 \liminf_{s\to\infty}\rho(s)\ge\lambda-\delta.
\]
Since \(\delta>0\) is arbitrary, together with \eqref{eq:r-limsup} we have
\(\rho(s)\to\lambda\).

Therefore, we have either \(r(x)\to \lambda\) or \(r(x)\to0\), as \(x\to 0\).  The first case says
\(u/W(u)\to \lambda\) and gives \eqref{eq:u0-linear}.  In the second case, we have
\(U'(x)\sim x/\lambda\), which implies
\(U(x)\sim x^2/(2\lambda)\), and this gives
\eqref{eq:u0-sharp}.
\end{proof}

Using Lemma~\ref{lem:u0-dichotomy}, we now determine the asymptotics of
$W_\lambda(u)$ at $u=0$ for every $\lambda\in\mathcal A$.
The critical asymptotic agrees with
\cite[Theorem~3.3(b)]{EnguicaGavioliSanchez2013}.
For $\lambda>\lambda_*$, \cite[Theorem~3.3(a)]{EnguicaGavioliSanchez2013}
gives $W_\lambda(u)=o(\sqrt u)$ as $u\downarrow0$, whereas we prove the
sharper asymptotic $W_\lambda(u)\sim u/\lambda$.
See also the related travelling-wave classification and phase-plane analysis
in~\cite{AudritoVazquez2017}.

\begin{proposition}
\label{prop:critical-branches}
At \(\lambda = \lambda_*\), we have
\begin{equation}
 W_{\lambda_*}(u)\sim\sqrt{2\lambda_*u},
 \qquad u\downarrow0.
 \label{eq:critical-sharp-branch}
\end{equation}
For every \(\lambda>\lambda_*\), we have
\begin{equation}
 W_\lambda(u)\sim\frac u\lambda,
 \qquad u\downarrow0.
 \label{eq:supercritical-linear-branch}
\end{equation}
\end{proposition}

\begin{proof}
We first show that an admissible parameter $\lambda\in\mathcal A$ with
linear asymptotics cannot be the left endpoint $\lambda_*\in\mathcal A$.
Suppose $\lambda_0\in\mathcal A$ and
\(W_{\lambda_0}(u)\sim u/\lambda_0\) as \(u\downarrow 0\).
Choose some \(k>1/\lambda_0\). Then
\(\lambda_0k-1>0\). Choose a small \(\eta>0\) such that
\[
 W_{\lambda_0}(\eta)<k\eta,
 \qquad
 \eta|1-k^3|\le\frac{\lambda_0k-1}{4}.
\]
Fix this \(k\) and \(\eta\). By Lemma~\ref{lem:shooting-properties}, the map
\(\lambda\mapsto W_\lambda(\eta)\) is continuous, so we can choose \(\lambda'<\lambda_0\) sufficiently close
to \(\lambda_0\) such that
\[
 \lambda'k-1\ge\frac{\lambda_0k-1}{2},
 \qquad
 W_{\lambda'}(\eta)<k\eta.
\]
Now consider $w(t)=k(1-t)$ on $t\in[1-\eta,1]$.  Writing
$u=1-t$, we have
\[
 \mathcal R_{\lambda'}[w]
 =u(\lambda'k-1)+u^2(1-k^3)
 \ge\frac{\lambda_0k-1}{4}u\ge0.
\]
At \(t=1-\eta\), since \(W_{\lambda'}(\eta)<k\eta\), the comparison principle gives
\[
 0\le y_{\lambda'}(t)\le[k(1-t)]^3,
 \qquad 1-\eta\le t\le1.
\]
In particular, \(y_{\lambda'}(1)=0\), so \(\lambda'\in\mathcal A\).  We have
therefore found an admissible parameter strictly smaller than \(\lambda_0\).
Thus $\lambda_*$ cannot have linear asymptotics.  Otherwise, one could
find $\lambda'<\lambda_*$ with $\lambda'\in\mathcal A$, contradicting
the definition of $\lambda_*$.  The dichotomy in
Lemma~\ref{lem:u0-dichotomy} then gives
\eqref{eq:critical-sharp-branch}.

We then prove \eqref{eq:supercritical-linear-branch}.
If \(\lambda>\lambda_*\), Lemma~\ref{lem:shooting-properties} gives
$W_\lambda(u)<W_{\lambda_*}(u)$ for $0<u<1$.  If
$W_\lambda(u)\sim\sqrt{2\lambda u}$, then
\[
 \frac{W_\lambda(u)}{W_{\lambda_*}(u)}\longrightarrow
 \sqrt{\frac \lambda{\lambda_*}}>1
\]
as \(u\downarrow0\), a contradiction.  Thus
\eqref{eq:supercritical-linear-branch} holds.
\end{proof}

\section{Proofs of Propositions~\ref{prop:full-line-distribution} and \ref{prop:hard-edge-distribution} }
\label{app:regularity}

This appendix has three components: Lemma~\ref{lem:finite-asymptotic-ode}
provides an abstract finite asymptotic expansion,
Lemma~\ref{lem:linear-branch-regularity} transfers it to the linear endpoint
branches of $W_\lambda$, and the final two proofs convert those endpoint
estimates into the two CDFs required in
Propositions~\ref{prop:full-line-distribution}
and~\ref{prop:hard-edge-distribution}.

\begin{lemma}
\label{lem:finite-asymptotic-ode}
Let $I\subset\mathbb R$ be an open interval containing $y_*$, and let
$F,G\in C^9(I)$ satisfy
\[
 F(y_*)=0,
 \qquad
 \mu:=F'(y_*)\ne0.
\]
Suppose that, for some $T_0>0$, a function $Y\in C^1([T_0,\infty);I)$
satisfies
\begin{equation}
 Y'(t)=F(Y(t))+\frac1tG(Y(t)),
 \qquad
 Y(t)\longrightarrow y_*
 \quad\text{as }t\to\infty.
 \label{eq:asymptotically-autonomous-ode}
\end{equation}
Then there exist coefficients $b_1,\ldots,b_4$, with
\[
 b_1=-\frac{G(y_*)}{F'(y_*)},
\]
such that, for $0\le \ell\le4$,
\begin{equation}
 \frac{\dd^\ell}{\dd t^\ell}
 \left(Y(t)-y_*-\sum_{j=1}^4b_jt^{-j}\right)
 =O(t^{-5-\ell})
 \quad\text{as }t\to\infty.
 \label{eq:finite-asymptotic-expansion}
\end{equation}
\end{lemma}

\begin{proof}
Equation~\eqref{eq:asymptotically-autonomous-ode} and the regularity of
$F$ and $G$ imply, by induction, that $Y$ has all derivatives required below
on $(T_0,\infty)$.  We first construct an expansion with eight terms.  Let
\[
\begin{aligned}
 p_8(x)&=\sum_{j=1}^8 b_jx^j,\\
 P_8(t)&=y_*+p_8(t^{-1}),\\
 D_8(t)&=P_8'(t)-F(P_8(t))-t^{-1}G(P_8(t)).
\end{aligned}
\]
Here $D_8$ is the defect of $P_8$.
With $x=t^{-1}$, this defect can be written as $D_8(t)=H_8(x)$, where
\[
 H_8(x)=-x^2p_8'(x)-F(y_*+p_8(x))-xG(y_*+p_8(x)).
\]
At order $x^j$, the coefficient $b_j$ enters the Taylor expansion of
$H_8$ only through the term $-\mu b_jx^j$; all remaining terms depend only
on $b_1,\ldots,b_{j-1}$.  Since $\mu\ne0$, the coefficients
$b_1,\ldots,b_8$ may therefore be chosen recursively so that
\[
 H_8^{(j)}(0)=0,
 \qquad 0\le j\le8.
\]
The coefficient of $x$ is $-\mu b_1-G(y_*)$, which gives the stated value of
$b_1$.  Since $H_8\in C^9$ on an interval containing the origin, Taylor's
theorem applied to $H_8^{(m)}$ gives
$H_8^{(m)}(x)=O(x^{9-m})$ for $0\le m\le4$.  Repeated use of
$\dd x/\dd t=-x^2$ then yields
\begin{equation}
 D_8^{(\ell)}(t)=O(t^{-9-\ell}),
 \qquad 0\le\ell\le4.
 \label{eq:D8-bound}
\end{equation}

Set $R=Y-P_8$.  For all sufficiently large $t$, both $Y(t)$ and $P_8(t)$
belong to a compact subinterval of $I$.  The mean-value formula gives the
exact equation
\begin{equation}
 R'(t)=A(t)R(t)-D_8(t),
 \label{eq:R-linear-equation}
\end{equation}
where
\[
 A(t)=\int_0^1\left[
 F'(P_8(t)+\theta R(t))
 +t^{-1}G'(P_8(t)+\theta R(t))
 \right]\dd\theta.
\]
Since $Y(t),P_8(t)\to y_*$, one has $A(t)\to\mu$.

Assume first that $\mu<0$.  Choose $T\ge T_0$ and $\alpha>0$ such that
$A(t)\le-\alpha$ for $t\ge T$.  Variation of constants gives
\begin{equation*}
 R(t)=\exp\left\{\int_T^tA(r)\,\dd r\right\}R(T)
 -\int_T^t\exp\left\{\int_s^tA(r)\,\dd r\right\}D_8(s)\,\dd s.
\end{equation*}
The first term is exponentially small.  For the integral term, split the
integral at $t/2$.  For $t\ge2T$, \eqref{eq:D8-bound} gives
\[
 \int_T^{t/2}e^{-\alpha(t-s)}|D_8(s)|\,\dd s
 \le e^{-\alpha t/2}\int_T^\infty|D_8(s)|\,\dd s
 =O(e^{-\alpha t/2}).
\]
On $[t/2,t]$, one has $s^{-9}\le2^9t^{-9}$, and hence
\[
\begin{aligned}
 \int_{t/2}^t e^{-\alpha(t-s)}s^{-9}\,\dd s
 \le 2^9t^{-9}
   \int_{t/2}^t e^{-\alpha(t-s)}\,\dd s
 \le \frac{2^9}{\alpha}t^{-9}.
\end{aligned}
\]
Together with \eqref{eq:D8-bound}, these estimates prove
$R(t)=O(t^{-9})$.

Assume now that $\mu>0$.  Increase $T$ so that $A(t)\ge\alpha>0$ for
$t\ge T$. Variation of constants again gives
\[
 R(t)=e^{-\int_t^SA(r)\,\dd r}R(S)
 +\int_t^S e^{-\int_t^sA(r)\,\dd r}D_8(s)\,\dd s.
\]
Since $R(S)\to0$, passing to the limit $S\to\infty$ yields
\begin{equation}
 R(t)=\int_t^\infty
 e^{-\int_t^sA(r)\,\dd r}D_8(s)\,\dd s.
 \label{eq:terminal-integral}
\end{equation}
The bound $A\ge\alpha$ and \eqref{eq:D8-bound} give
$R(t)=O(t^{-9})$.

It remains to control derivatives of $R$.  Suppose inductively that
$R^{(j)}(t)=O(t^{-9})$ for $0\le j\le m$, where $m\le3$.
Differentiation of the displayed formula for $A$ shows that
$A^{(j)}(t)=O(1)$ for $0\le j\le m$: every term contains only bounded
derivatives of $F$ and $G$, derivatives of the explicit function $P_8$, and
derivatives of $R$ already controlled by the induction hypothesis.
Differentiating \eqref{eq:R-linear-equation} $m$ times therefore gives
\[
 R^{(m+1)}(t)
 =\sum_{j=0}^{m}\binom mj A^{(j)}(t)R^{(m-j)}(t)
  -D_8^{(m)}(t)
 =O(t^{-9}).
\]
Thus $R^{(\ell)}(t)=O(t^{-9})$ for $0\le\ell\le4$.  Finally,
\[
 Y-y_*-\sum_{j=1}^4b_jt^{-j}
 =\sum_{j=5}^8b_jt^{-j}+R,
\]
and differentiating this identity proves
\eqref{eq:finite-asymptotic-expansion}.
\end{proof}

\begin{lemma}
\label{lem:linear-branch-regularity}
Let $W$ be a positive solution of \eqref{eq:W-ode} with parameter
$\lambda>0$.

If $W(u)\sim u/\lambda$ as $u\downarrow0$, then $W$ extends to a
$C^4$ function on $[0,1)$ and
\begin{equation*}
 W(u)=\frac u\lambda+
 \left(\lambda^{-4}-\lambda^{-1}\right)u^2+O(u^3)
 \quad\text{as }u\downarrow0.
\end{equation*}
More precisely, there exist $\eta>0$, coefficients $c_3,c_4,c_5$, and a
function $\rho_0\in C^4((0,\eta])$ such that
\begin{samepage}
\begin{equation}
 W(u)=\frac u\lambda+
 \left(\lambda^{-4}-\lambda^{-1}\right)u^2
 +\sum_{j=3}^5c_ju^j+\rho_0(u),
 \qquad
 \rho_0^{(k)}(u)=O(u^{6-k})
 \label{eq:W-expansion-zero-differentiable}
\end{equation}
for $0\le k\le4$.
\end{samepage}

If $W(u)\sim(1-u)/\lambda$ as $u\uparrow1$, then $W$ extends to a
$C^4$ function on $(0,1]$ and
\begin{equation*}
 W(u)=\frac{1-u}{\lambda}
 -\left(\lambda^{-4}+\lambda^{-1}\right)(1-u)^2
 +O((1-u)^3)
 \quad\text{as }u\uparrow1.
\end{equation*}
More precisely, there exist $\eta>0$, coefficients
$\widehat c_3,\widehat c_4,\widehat c_5$, and a function
$\rho_1\in C^4((0,\eta])$ such that, with $r=1-u$,
\begin{samepage}
 \begin{equation}
 W(1-r)=\frac r\lambda
 -\left(\lambda^{-4}+\lambda^{-1}\right)r^2
 +\sum_{j=3}^5\widehat c_jr^j+\rho_1(r),
 \qquad
 \rho_1^{(k)}(r)=O(r^{6-k})
 \label{eq:W-expansion-one-differentiable}
\end{equation}
for $0\le k\le4$.
\end{samepage}

At each endpoint for which the corresponding linear asymptotic holds, there
exist $\eta>0$ and $C<\infty$ such that
\begin{equation}
 |W'(u)|+|W(u)W''(u)|+|W(u)^2W'''(u)|\le C
 \label{eq:W-relative-derivative-bounds}
\end{equation}
throughout the associated one-sided interval $(0,\eta]$ or
$[1-\eta,1)$.
\end{lemma}

\begin{proof}
Consider first the endpoint $u=0$ and set
\[
 Y(t)=tW(t^{-1}),
 \qquad t>1.
\]
Substitution of $W(u)=uY(u^{-1})$ into \eqref{eq:W-ode} gives
\begin{equation*}
 Y'(t)=-\frac{\lambda Y(t)-1}{Y(t)^2}
       -\frac{1-Y(t)^3}{tY(t)^2},
 \qquad
 Y(t)\longrightarrow \lambda^{-1}.
\end{equation*}
Thus Lemma~\ref{lem:finite-asymptotic-ode} applies with
\[
 F_0(y)=\frac{1-\lambda y}{y^2},
 \qquad
 G_0(y)=y-y^{-2},
 \qquad
 y_*=\lambda^{-1}.
\]
Since
\[
 F_0'(\lambda^{-1})=-\lambda^3,
 \qquad
 -\frac{G_0(\lambda^{-1})}{F_0'(\lambda^{-1})}
 =\lambda^{-4}-\lambda^{-1},
\]
there are coefficients $b_2,b_3,b_4$ and a remainder $R_0$ such that
\(
 R_0^{(k)}(t)=O(t^{-5-k}),
 0\le k\le4,
\)
and
\[
 Y(t)=\lambda^{-1}
 +\left(\lambda^{-4}-\lambda^{-1}\right)t^{-1}
 +\sum_{j=2}^4b_jt^{-j}+R_0(t).
\]
Multiplication by $u=t^{-1}$ gives
\eqref{eq:W-expansion-zero-differentiable}, once the transformed remainder
is shown to have the stated derivative bounds.

For the endpoint $u=1$, set
\[
 \widehat Y(t)=tW(1-t^{-1}),
 \qquad t>1.
\]
Then
\begin{equation*}
 \widehat Y'(t)=\frac{\lambda\widehat Y(t)-1}{\widehat Y(t)^2}
 +\frac{1+\widehat Y(t)^3}{t\widehat Y(t)^2},
 \qquad
 \widehat Y(t)\longrightarrow\lambda^{-1}.
\end{equation*}
Here Lemma~\ref{lem:finite-asymptotic-ode} applies with
\[
 F_1(y)=\frac{\lambda y-1}{y^2},
 \qquad
 G_1(y)=y+y^{-2}.
\]
Moreover,
\[
 F_1'(\lambda^{-1})=\lambda^3,
 \qquad
 -\frac{G_1(\lambda^{-1})}{F_1'(\lambda^{-1})}
 =-(\lambda^{-4}+\lambda^{-1}).
\]
The case $F_1'(\lambda^{-1})>0$ is covered by the terminal representation
\eqref{eq:terminal-integral}.  It yields
\eqref{eq:W-expansion-one-differentiable} after multiplication by
$r=t^{-1}=1-u$.

We record explicitly the change of variables used for both remainders.  Let
$R$ be a function satisfying
\[
 \frac{\dd^kR}{\dd t^k}(t)=O(t^{-5-k}),
 \qquad 0\le k\le4,
\]
and define $E(r)=rR(r^{-1})$ for $r>0$.  Put $t=r^{-1}$.
Derivatives of $E$ in the following display are taken with respect to $r$,
whereas derivatives of $R$ are taken with respect to $t$.  Since
$\dd t/\dd r=-t^2$, the chain rule gives
\begin{samepage}
 \begin{align*}
  \frac{\dd E}{\dd r}(r)&=R(t)-t\frac{\dd R}{\dd t}(t),
 &&\frac{\dd^2E}{\dd r^2}(r)=t^3\frac{\dd^2R}{\dd t^2}(t), \\
  \frac{\dd^3E}{\dd r^3}(r)&=-3t^4\frac{\dd^2R}{\dd t^2}(t)-t^5\frac{\dd^3R}{\dd t^3}(t),
  &&\frac{\dd^4E}{\dd r^4}(r)=12t^5\frac{\dd^2R}{\dd t^2}(t)+8t^6\frac{\dd^3R}{\dd t^3}(t)+t^7\frac{\dd^4R}{\dd t^4}(t).
\end{align*}
\end{samepage}

Consequently,
\[
 \frac{\dd^kE}{\dd r^k}(r)=O(r^{6-k}),
 \qquad 0\le k\le4.
\]
Thus the transformed remainder, together with its first four derivatives,
extends continuously to $r=0$ with value zero.  Taking $r=u$ at the left
endpoint and $r=1-u$ at the right endpoint proves the asserted $C^4$
extensions.  The expansions then imply \eqref{eq:W-relative-derivative-bounds}
on the corresponding one-sided intervals.
\end{proof}

\begin{proof}[Proof of Proposition~\ref{prop:full-line-distribution}]
By Proposition~\ref{prop:critical-branches} and
Lemma~\ref{lem:linear-u1}, both endpoint branches of $W_{\lambda}$ are
linear.  Lemma~\ref{lem:linear-branch-regularity} therefore applies at
$u=0$ and $u=1$.

Define
\begin{equation*}
 \Theta_{\lambda}(v)=\int_{1/2}^{v}
 \frac{\dd s}{\beta W_{\lambda}(1-s)},
 \qquad 0<v<1.
\end{equation*}
As $s\downarrow0$,
$W_{\lambda}(1-s)\sim s/\lambda$, whereas as $s\uparrow1$,
$W_{\lambda}(1-s)\sim(1-s)/\lambda$.  Hence
\[
 \Theta_{\lambda}(v)\longrightarrow-\infty
 \quad\text{as }v\downarrow0,
 \qquad
 \Theta_{\lambda}(v)\longrightarrow\infty
 \quad\text{as }v\uparrow1.
\]
Since $\Theta_{\lambda}'(v)>0$, the map $\Theta_{\lambda}$ is a smooth bijection from
$(0,1)$ onto $\mathbb R$.  Write
$\Phi=\Phi_\lambda:=\Theta_{\lambda}^{-1}$.  Then \(\Phi\) satisfies
\eqref{eq:Phi-phase-definition}.  Conversely, separation of variables
applied to the differential equation in \eqref{eq:Phi-phase-definition}
shows that every increasing solution satisfies
$\Theta_{\lambda}(\Phi(z))=z-z_0$ for some constant $z_0$.  Thus any two such
solutions differ by a translation, and the normalization $\Phi(0)=1/2$
fixes that translation.

Put $q=q_\lambda=\Phi'$ and $u=1-\Phi(z)$.  Then
$q=\beta W_{\lambda}(u)$ and
$u'=-\beta W_{\lambda}(u)$.  In particular,
\[
 q'=-\beta^2W_{\lambda}(u)W_{\lambda}'(u).
\]
Using $a\beta^3=2$ and \eqref{eq:W-ode}, one obtains
\[
 a\Phi'\Phi''
 =-2W_{\lambda}(u)^2W_{\lambda}'(u)
 =-2\lambda W_{\lambda}(u)+2u(1-u)
 =-\kappa_{\lambda}q+2\Phi(1-\Phi),
\]
which proves \eqref{eq:upper-profile-equation}.

All occurrences of $W_{\lambda}$ and its derivatives in the following identities
are evaluated at $u=1-\Phi(z)$.  Repeated differentiation gives
\[
\begin{aligned}
 \frac{q'}q
 =-\beta W_{\lambda}',\qquad
 \frac{q''}q
 =\beta^2\bigl((W_{\lambda}')^2
   +W_{\lambda}W_{\lambda}''\bigr),\\
 \frac{q'''}q
 =-\beta^3\bigl((W_{\lambda}')^3
   +4W_{\lambda}W_{\lambda}'W_{\lambda}''
   +W_{\lambda}^2W_{\lambda}'''\bigr).
\end{aligned}
\]
The bounds in \eqref{eq:W-relative-derivative-bounds} control these
expressions on one-sided intervals at $0$ and $1$.  On the complementary
compact subinterval of $(0,1)$ they are bounded by continuity.  This proves
\eqref{eq:full-line-relative-bounds}.

As $z\to-\infty$, one has $\Phi(z)\to0$ and
$u=1-\Phi(z)\to1$.  The linear expansion at $u=1$ therefore gives
\[
 \frac{q(z)}{\Phi(z)}
 =\beta\frac{W_{\lambda}(1-\Phi(z))}{\Phi(z)}
 \longrightarrow\frac\beta\lambda.
\]
Similarly, as $z\to\infty$, the linear expansion at $u=0$ gives
\[
 \frac{q(z)}{1-\Phi(z)}
 =\beta\frac{W_{\lambda}(u)}u
 \longrightarrow\frac\beta\lambda.
\]
This proves \eqref{eq:full-line-tail-ratios}.

Choose $L>0$ such that both ratios in
\eqref{eq:full-line-tail-ratios} are at least
$c:=\beta/(2\lambda)$ on their respective tails.  Then
\[
 (\log\Phi)'=\frac q\Phi\ge c
 \quad\text{for }z\le-L,
 \qquad
 -\bigl(\log(1-\Phi)\bigr)'=\frac q{1-\Phi}\ge c
 \quad\text{for }z\ge L.
\]
Integration gives the stated exponential bounds after increasing the
multiplicative constant on the compact interval $[-L,L]$.  Finally, the
tail-integral identity
\[
 \int_{\mathbb R}|x|\,\dd\Phi(x)
 =\int_{-\infty}^0\Phi(x)\,\dd x
  +\int_0^\infty(1-\Phi(x))\,\dd x
\]
shows that the absolute first moment is finite.
\end{proof}

\begin{proof}[Proof of Proposition~\ref{prop:hard-edge-distribution}]
Proposition~\ref{prop:subcritical-W} gives $W_{\lambda}(0)>0$, and
Lemma~\ref{lem:linear-u1} gives
$W_{\lambda}(u)\sim(1-u)/\lambda$ as $u\uparrow1$.  Define
\[
 z(v)=\int_0^v\frac{\dd s}{\beta W_{\lambda}(s)},
 \qquad 0\le v<1.
\]
The function $z(v)$ is continuous at $v=0$, while the linear expansion at $s=1$
implies $z(v)\to\infty$ as $v\uparrow1$.  Since $z'(v)>0$, the map $z$ is
a bijection from $[0,1)$ onto $[0,\infty)$.  Its inverse is the function
$\Psi_{\lambda}$ in \eqref{eq:Psi-subcritical-definition}.  Since
$W_{\lambda}$ is smooth and strictly positive on $[0,1)$, with
$W_{\lambda}(0)>0$, the one-sided inverse function theorem shows that
$\Psi_{\lambda}$ is smooth on $[0,\infty)$.  Separation of variables gives
uniqueness.

For $z\ge0$, write $v=\Psi_{\lambda}(z)$.  Then
$q_{\lambda}=\beta W_{\lambda}(v)$ and
\[
 q_{\lambda}'=\beta^2W_{\lambda}(v)W_{\lambda}'(v).
\]
Since $a\beta^3=2$, equation \eqref{eq:W-ode} gives
\[
 a q_{\lambda}q_{\lambda}'
 =2W_{\lambda}(v)^2W_{\lambda}'(v)
 =2\lambda W_{\lambda}(v)-2v(1-v)
 =\kappa_{\lambda}q_{\lambda}
  -2\Psi_{\lambda}(1-\Psi_{\lambda}),
\]
which proves \eqref{eq:lower-profile-equation}, including the value at
$z=0$ by continuity.

All occurrences of $W_{\lambda}$ and its derivatives below are evaluated at
$v=\Psi_{\lambda}(z)$.  Repeated differentiation gives
\[
\begin{aligned}
 \frac{q_{\lambda}'}{q_{\lambda}}
 =\beta W_{\lambda}',\qquad
 \frac{q_{\lambda}''}{q_{\lambda}}
 =\beta^2\bigl((W_{\lambda}')^2
   +W_{\lambda}W_{\lambda}''\bigr),\qquad
 \frac{q_{\lambda}'''}{q_{\lambda}}
 =\beta^3\bigl((W_{\lambda}')^3
   +4W_{\lambda}W_{\lambda}'W_{\lambda}''
   +W_{\lambda}^2W_{\lambda}'''\bigr).
\end{aligned}
\]
On an interval $[0,\eta]$, Proposition~\ref{prop:subcritical-W} implies that
$W_{\lambda}$ is smooth and bounded away from zero.  On an interval
$[1-\eta,1)$, Lemma~\ref{lem:linear-branch-regularity} gives
\eqref{eq:W-relative-derivative-bounds}.  Continuity controls the remaining
compact interval.  Hence \eqref{eq:subcritical-relative-bounds} follows.
The linear expansion at $v=1$ also gives
\eqref{eq:subcritical-right-tail}.

Choose $L>0$ such that
$q_{\lambda}(z)/(1-\Psi_{\lambda}(z))\ge
c:=\beta/(2\lambda)$ for $z\ge L$.  Then
\[
 -\frac{\dd}{\dd z}\log(1-\Psi_{\lambda}(z))
 =\frac{q_{\lambda}(z)}{1-\Psi_{\lambda}(z)}\ge c,
 \qquad z\ge L.
\]
Integration gives the stated exponential bound.  Since the distribution is
supported on $[0,\infty)$,
\[
 \int_{[0,\infty)}x\,\dd\Psi_{\lambda}(x)
 =\int_0^\infty(1-\Psi_{\lambda}(x))\,\dd x<\infty.
\]

It remains to construct the extension of $q_{\lambda}$.  Set
$h=\log q_{\lambda}$ on $[0,\infty)$.  From
\eqref{eq:subcritical-relative-bounds},
\[
\begin{aligned}
 h'&=\frac{q_{\lambda}'}{q_{\lambda}},\qquad
 h''&=\frac{q_{\lambda}''}{q_{\lambda}}
      -\left(\frac{q_{\lambda}'}{q_{\lambda}}\right)^2,\qquad
 h'''&=\frac{q_{\lambda}'''}{q_{\lambda}}
      -3\frac{q_{\lambda}'q_{\lambda}''}{q_{\lambda}^2}
      +2\left(\frac{q_{\lambda}'}{q_{\lambda}}\right)^3
\end{aligned}
\]
are bounded on $[0,\infty)$.

 Let
\[
 P(z)=\sum_{j=0}^3\frac{h^{(j)}(0)}{j!}z^j,
\]
and choose $\chi\in C^\infty(\mathbb R)$ such that
$\chi=1$ on $[-1/2,0]$ and $\chi=0$ on $(-\infty,-1]$.  Define
\[
 \overline h(z)=
 \begin{cases}
  h(z),&z\ge0,\\
  \chi(z)P(z)+(1-\chi(z))h(0),&z<0.
 \end{cases}
\]
Then $\overline h\in C^3(\mathbb R)$, it agrees with $h$ on
$[0,\infty)$, and its first three derivatives are bounded.  Set
$\overline q_{\lambda}=e^{\overline h}$.  The identities
\[
 {
 \frac{\overline q_{\lambda}'}{\overline q_{\lambda}}
 =\overline h',
 \qquad
 \frac{\overline q_{\lambda}''}{\overline q_{\lambda}}
 =\overline h''+(\overline h')^2,
 \qquad
 \frac{\overline q_{\lambda}'''}{\overline q_{\lambda}}
 =\overline h'''+3\overline h'\overline h''+(\overline h')^3}
\]
prove \eqref{eq:q-extension-bounds}.
\end{proof}

\section*{Code availability}
The Lean~4 development is publicly available at
\url{https://github.com/iyalice/Series-Parallel_lean}.
The accompanying software release is \texttt{v1.0.3}, available at
\url{https://github.com/iyalice/Series-Parallel_lean/releases/tag/v1.0.3},
with source archive
\url{https://github.com/iyalice/Series-Parallel_lean/archive/refs/tags/v1.0.3.zip}.
The release corresponds to commit
\texttt{9eb6338993276b2c2783\allowbreak{}b2b56236d7ee1badb91f}.
It contains the reference manuscript, manuscript--Lean correspondence,
source map, build and axiom reports, and reproduction instructions.
The software is cited in~\cite{DingHeLiangZheng2026Software}.
The development uses Lean~4.32.1, Lake~5.0.0, and mathlib~v4.32.1 at commit
\texttt{520045ab14e26149ee97\allowbreak{}0e2e617ca04b09bde5d6}, under the Apache-2.0 license.

\section*{Declaration of generative AI and AI-assisted technologies in the manuscript preparation process}

During the preparation of this manuscript, the authors used OpenAI
ChatGPT and Codex to assist with language editing, organization, and
clarity of presentation.  Their use in the research and formalization
process is described separately in
Section~\ref{subsec:ai-methodology}.  The authors reviewed and edited
all resulting text and take full responsibility for the content of
the article.


\begin{thebibliography}{99}

\bibitem{AddarioBerryEtAl2020}
L.~Addario-Berry, H.~Cairns, L.~Devroye, C.~Kerriou, and R.~Mitchell,
\emph{Hipster random walks},
Probab. Theory Related Fields \textbf{178} (2020), 437--473.
\url{https://doi.org/10.1007/s00440-020-00980-z}

\bibitem{AddarioBerryBeckmanLin2022}
L.~Addario-Berry, E.~Beckman, and J.~Lin,
\emph{Asymmetric cooperative motion in one dimension},
Trans. Amer. Math. Soc. \textbf{375} (2022), no.~4, 2883--2913.
\url{https://doi.org/10.1090/tran/8581}

\bibitem{AddarioBerryBeckmanLin2024}
L.~Addario-Berry, E.~Beckman, and J.~Lin,
\emph{Symmetric cooperative motion in one dimension},
Probab. Theory Related Fields \textbf{188} (2024), 625--666.
\url{https://doi.org/10.1007/s00440-023-01244-2}

\bibitem{AldousBandyopadhyay2005}
D.~J. Aldous and A.~Bandyopadhyay,
\emph{A survey of max-type recursive distributional equations},
Ann. Appl. Probab. \textbf{15} (2005), no.~2, 1047--1110.
\url{https://doi.org/10.1214/105051605000000142}

\bibitem{AudritoVazquez2017}
A.~Audrito and J.~L. V\'azquez,
\emph{The Fisher--KPP problem with doubly nonlinear diffusion},
J. Differential Equations \textbf{263} (2017), 7647--7708.
\url{https://doi.org/10.1016/j.jde.2017.08.025}

\bibitem{AuffingerCable2017}
A.~Auffinger and D.~Cable,
\emph{Pemantle's min-plus binary tree},
arXiv:1709.07849 [math.PR], 2017.
\url{https://doi.org/10.48550/arXiv.1709.07849}

\bibitem{BarlesSouganidis1991}
G.~Barles and P.~E. Souganidis,
\emph{Convergence of approximation schemes for fully nonlinear second order
equations},
Asymptotic Anal. \textbf{4} (1991), no.~3, 271--283.
\url{https://doi.org/10.3233/ASY-1991-4305}

\bibitem{ChenDerridaDuquesneShi2026}
X.~Chen, B.~Derrida, T.~Duquesne, and Z.~Shi,
\emph{The distance on the slightly supercritical random series--parallel graph},
Adv. Appl. Probab. \textbf{58} (2026), no.~1, 80--121.
\url{https://doi.org/10.1017/apr.2025.10023}

\bibitem{ChenDuquesneShi2026}
X.~Chen, T.~Duquesne, and Z.~Shi,
\emph{Hipster random walks, random series--parallel graph and random homogeneous
systems},
arXiv:2511.16880v2 [math.PR], 2026.
\url{https://arxiv.org/abs/2511.16880v2}

\bibitem{DingHeLiangZheng2026Software}
R.~Ding, Z.~He, Y.~Liang, and Y.~Zheng,
\emph{Lean~4 formalization for Distance and resistance on random series--parallel graphs},
software release v1.0.3, 2026.
\url{https://github.com/iyalice/Series-Parallel_lean/releases/tag/v1.0.3}



\bibitem{EnguicaGavioliSanchez2013}
R.~Engui\c{c}a, A.~Gavioli, and L.~Sanchez,
\emph{A class of singular first order differential equations with applications
in reaction--diffusion},
Discrete Contin. Dyn. Syst. \textbf{33} (2013), no.~1, 173--191.
\url{https://doi.org/10.3934/dcds.2013.33.173}



\bibitem{HadelerRothe1975}
K.~P. Hadeler and F.~Rothe,
\emph{Travelling fronts in nonlinear diffusion equations},
J. Math. Biol. \textbf{2} (1975), 251--263.
\url{https://doi.org/10.1007/BF00277154}

\bibitem{HamblyJordan2004}
B.~M. Hambly and J.~Jordan,
\emph{A random hierarchical lattice: the series--parallel graph and its
properties},
Adv. Appl. Probab. \textbf{36} (2004), no.~3, 824--838.
\url{https://doi.org/10.1239/aap/1093962236}

\bibitem{Jordan2002}
J.~Jordan,
\emph{Almost sure convergence for iterated functions of independent random
variables},
Ann. Appl. Probab. \textbf{12} (2002), no.~3, 985--1000.
\url{https://doi.org/10.1214/aoap/1031863178}

\bibitem{LiRogers1999}
D.~Li and T.~D. Rogers,
\emph{Asymptotic behavior for iterated functions of random variables},
Ann. Appl. Probab. \textbf{9} (1999), no.~4, 1175--1201.
\url{https://doi.org/10.1214/aoap/1029962869}

\bibitem{Morfe2026}
P.~S. Morfe,
\emph{Analysis of a class of recursive distributional equations including the
resistance of the series--parallel graph},
arXiv:2511.11036v3 [math.PR], 2026.
\url{https://arxiv.org/abs/2511.11036v3}

\bibitem{SanchezGardunoMaini1994}
F.~S\'anchez-Gardu\~no and P.~K. Maini,
\emph{Existence and uniqueness of a sharp travelling wave in degenerate
non-linear diffusion Fisher--KPP equations},
J. Math. Biol. \textbf{33} (1994), 163--192.
\url{https://doi.org/10.1007/BF00160178}

\bibitem{Shneiberg1986}
I.~Ya. Shneiberg,
\emph{Hierarchical sequences of random variables},
Theory Probab. Appl. \textbf{31} (1987), no.~1, 137--141.
\url{https://doi.org/10.1137/1131018}

\end{thebibliography}
\end{document}